\documentclass[a4paper]{amsart} 
\usepackage{amsmath}
\usepackage{amssymb}%This gives us symbols like \square
\usepackage{verbatim}%This gives us the \begin{comment} environment.
\usepackage{amsthm}%This gives us Definition style of theorem.
\usepackage{mathrsfs}%This gives us \mathscr#
\usepackage{graphicx}
\usepackage{mathtools}
\usepackage[colorlinks,pdfdisplaydoctitle,linkcolor=black,citecolor=red]{hyperref}
\usepackage{caption}
\usepackage{subcaption}
\usepackage{url}
\usepackage{epstopdf}
\usepackage{pgf,tikz}
\usepackage{bbm}
\usepackage{extarrows}
\usetikzlibrary{arrows}
\usepackage{algorithm}
 \usepackage{algorithmicx}
\usepackage{amssymb,amsthm}    % amssymb package contains more mathematical symbols
\usepackage{graphicx}          % graphicx package enables you to paste in graphics
\usepackage{amsmath}         % best maths package on offer
\usepackage{caption}
\usepackage{subcaption}
\usepackage{ dsfont }
\usepackage{listings}
\usepackage{bbm}
\usepackage{color}
\usepackage{ dsfont}
\usepackage{enumerate}
 \usepackage{relsize}
\usepackage{xfrac}

\renewcommand{\Delta}{\triangle}

\definecolor{darkblue}{rgb}{0,0,0.7}

\definecolor{darkgreen}{rgb}{0.01,0.75,0.24}

\def \Ee[#1]{\mathcal{E}^{\text{{#1}}}}

\def\pa[#1,#2]{\frac{\partial {#1}}{\partial {#2}} }

\def\idom[#1,#2,#3]{\int_{#1}\hspace{1pt} {#2} \hspace{1pt} \text{d}{#3}}
\def\res[#1,#2]{\left.{#1}\right|_{#2}}

\def\var[#1,#2]{\langle \delta \mathcal{E}^{\text{{#1}}}({#2}),v\rangle}
\def\vars[#1,#2,#3]{\langle \delta^2\mathcal{E}^{\text{{#1}}}({#2})v,{#3}\rangle}
\def\vard[#1,#2,#3,#4]{\langle \delta\mathcal{E}^{\text{{#1}}}({#2})-\delta\mathcal{E}^{\text{{#3}}}({#4}),v\rangle}

\newcommand{\be}{\begin{equation}}
\newcommand{\en}{\end{equation}}
\newcommand{\ben}{\begin{equation*}}
\newcommand{\enn}{\end{equation*}}
\newcommand{\bea}{\begin{aligned}}
\newcommand{\ena}{\end{aligned}}

\def\ba#1\ena{\begin{align}#1\end{align}}
\def\ban#1\enan{\begin{align*}#1\end{align*}}

\theoremstyle{plain}
\newtheorem{thm}{Theorem}[section]
\newtheorem{defn}[thm]{Definition}

\newtheorem{lem}[thm]{Lemma}
\newtheorem{cor}[thm]{Corollary}

\newtheorem{remark}[thm]{Remark}

\numberwithin{equation}{section}

\begin{document}
%\title{A Unified statistical framework of perfect radar pulse compression}
\title[Posterior convergence of $\mathlarger{\mathlarger{\mathlarger{\alpha}}}$-stable sheets]{Posterior convergence analysis of  $\mathlarger{\mathlarger{\mathlarger{\alpha}}}$-stable sheets}
%\title{A Rigorous statistical understanding of perfect radar pulse compression}
\author[N. K. Chada] {Neil K. Chada}
\address{Department of Statistics and Applied Probability, National University of Singapore, 119077, Singapore}
\email{neil.chada@nus.edu.sg}

\author[S. Lasanen] {Sari Lasanen}
\address{School of Engineering Science, Lappeenranta-Lahti University of Technology, Yliopistonkatu 34, FI-53850 Lappeenranta, Finland}
\email{sari.lasanen@lut.fi}

%\author[P. Piironinen] {Petteri Piironinen}
%\address{Department of Mathematics, University of Helsinki, FI-00014, Helsinki, Finland}
%\email{petteri.piiroinen@helsinki.fi}

\author[L. Roininen] {Lassi Roininen}
\address{School of Engineering Science, Lappeenranta-Lahti University of Technology,  Yliopistonkatu 34, FI-53850 Lappeenranta, Finland}
\email{lassi.roininen@lut.fi}

\maketitle

\begin{abstract}
{This paper is concerned with the theoretical understanding of $\alpha$-stable sheets $U$ on $\mathbb{R}^d$ within the framework of Bayesian inference. Our motivation is in the context of Bayesian inverse problems, where we consider the distributions of  random fields $U$  as qualitatively informative prior distributions.  We derive posterior convergence results referring to finite-dimensional approximations of  infinite-dimensional random variables $U$. In doing so we use a number of variants, which these sheets $U$  can take, such as a stochastic integral representation, but also a random series expansion, namely,  the L\'evy-LePage series. Especially,  we study natural discretizations of  $U$, which arise from independent stable  increments. Our proofs rely on the fact  that the random sheets $U$  omit $L^p$-sample paths.  Aside from convergence of approximations of stable sheets, we address whether both well-definedness and well-posedness of the inverse problem can be attained. }

%This paper is concerned with the theoretical understanding of $\alpha$-stable sheets $U$ on $\mathbb{R}^d$ within the framework of Bayesian inference. Our motivation is in the context of Bayesian inverse problems, where we consider the distributions of  random fields $U$  as qualitatively informative prior distributions.  We derive posterior convergence results referring to finite-dimensional approximations of  infinite-dimensional random variables $U$. In doing so we use a number of variants, which these sheets $U$  represent. Our proofs rely on the fact that the random sheets $U$  omit $L^p$-sample paths. We also address whether well-definedness and well-posedness of the inverse problem can be attained.
\end{abstract}
\tableofcontents
\noindent
{\textbf{AMS subject classifications:} 60G52, 60G17, 62M40, 35R30. $L^p$-convergence, well-definedness.}

       % e.g. \ams{60E20}{49G03;49F10}

\section{Introduction}
\label{sec:introduction}

Inverse problems  are concerned with the recovery of an underlying unknown from data, represented by noisy measurements of a model solution. One way to approximate solutions to an inverse problem is to treat the unknown as a probabilistic distribution, referred to as the posterior, which is known as a statistical or Bayesian approach \cite{KS04,FS86,AMS10}.

 Since the development of inverse problems in the Bayesian setting, an important and fundamental question is how to develop informative priors. Extending the family of numerically applicable prior distributions is of  crucial importance as the posterior inherits properties of the prior, when the measurements do not fully determine the unknown.   In the beginning of the millennium   the  development of  priors that mimic,  in various ways,   the behaviour of discontinuous  unknowns in inverse problems,  was initiated as a part of this extension process.   We emphasize that we will only discuss below the priors for  inverse problems, as ordinary regression problems do not usually need  as strong  prior assumptions as the severely ill-posed inverse problems.

\subsection{Related work}

Extending the family of numerically applicable prior distributions  is
 of  crucial importance as the posterior inherits properties of the prior, when the measurements do not fully determine the unknown.   In the beginning of the millennium   the  development of  priors that mimic,  in various ways,   the behaviour of discontinuous  unknowns in inverse problems,  was initiated as a part of this extension process.   We emphasize that we will only discuss below the priors for  inverse problems, as ordinary regression problems do not usually need  as strong  prior assumptions as the severely ill-posed inverse problems.
 
  One of the first edge-preserving priors that was used in the context of parametric Bayesian inversion was the total variation (TV) prior. The motivation for the prior was naturally derived  from TV regularization in non-statistical inverse problems and the success of LASSO methods in regression problems.   The first systematic approach to TV priors was the  study by Lassas and Siltanen in 2004 \cite{LS04}, who showed that the TV prior was degenerate to mesh refinements. As a result the prior was viewed not a sensible choice to use in the Bayesian approach. 

As a remedy, Lassas et al.\ \cite{LSS09} formulated in 2009 the first nonparametric edge-preserving prior by introducing Besov space-valued random series expansions. Shortly after, a  regularization motivated approach was made by Helin  \cite{Helin09}.  Helin's hierarchical Gaussian prior is linked to approximations of  the Mumford-Shah (M-S) functional. The M-S functional allows discontinuities of the unknown over segments, and, more importantly, penalizes  the size of  the singular sets of the unknown.  Helin's hierarchical prior  enjoys numerical ease  due to use of hierarchical Gaussian distributions, but the prior is only approximatively  edge-preserving as the applied Ambrosio-Tortorelli approximation promotes continuous sample paths through the introduction of an additional sharpness parameter.  The benefit of these priors over TV, is that they remain discretization-invariant, which is an attractive property to have in numerical approximations.

 The Besov prior has seen substantial progress since its formulation. Notably the well-posedness of the Bayesian inverse problem in  Hölder spaces   has been shown  by Dashti et al.\cite{DHS12}  for regular enough Besov priors, where the regularity is in essence considered  in terms of the maximum value and decay speed  of the  expected  absolute values of  wavelet coefficients.   Further enhancements  towards  maximum a-posteriori inversion has been done in \cite{ABDH17}.

The prior distribution recipe of Lassas  of et al.\  of defining  random functions with the help of just  regular enough deterministic functions and independent random  coefficients,  which have distributions, whose track record is good in Bayesian inference,  has since been successfully adapted to other cases.  
Whereas Lassas et al.\ used wavelet expansions and coefficients with generalized Laplace distributions
(known as exponential power distributions \cite{BoxTiao62}), others have concentrated on unconditional Schauder basis in Banach spaces and 
heavy-tailed distributions, which produce prior distributions  that promote small number of large local variations.

Sullivan studied well-posedness of the Bayesian inverse problem in quasi-Banach spaces with $\alpha$-stable  coefficients \cite{TJS16}. Well-posedness for  stable prior distributions is a challenging problem since the prior is only  $p$-integrable instead of 
exponentially integrable as is the case in Stuart's classical approach \cite{AMS10}.  In Sullivan's method  the  key element for well-posedness is to restrict to  forward mappings that have  slowly growing lower bounds.

Hosseini \cite{BH17,HN17}  suggested several types of priors from  random series with a Schauder basis and  generalized Gamma distributed coefficients to pure jump Lévy processes, namely, compound Poisson processes and their generalizations, randomly truncated random series in Banach spaces.  Hosseini's study  \cite{BH17}  is the first systematic treatment of prior and posterior distributions on the important nonseparable function space $BV(D)$ of  functions of bounded variation on an open set $D\subset \mathbb R$. Among the results shown by Hosseini is  a significant technical detail  about the Radon property of distributions, of compound Poisson processes, on $BV(D)$. In \cite{BH17},   well-posedness for the heavy-tailed distributions follows   similar lines as in Sullivan's approach  \cite{TJS16}.

More recently, aspects of deep learning have also been applied to Bayesian inversion, such as in the context of deep Gaussian processes. The paper of Dunlop et al. \cite{DGST19} discuss how much depth is required from a deep Gaussian process  (DGP) prior to capture features of the unknown. Their results suggest the DGP prior does not need to be very deep while another contribution is that it offers flexibility in modelling. yet the main theory derived for these priors has been on ergodicity, as remaining theory for inversion holds in the same way as for Gaussian priors. %However in terms of theoretical gains, most of what has been done is quantifying well-posedness.  

Although truncations of the random series expansions provide natural parametric approximations of the corresponding random functions, these approximations can suffer from an  intrinsic preference of the magnitude of  jumps  at different location. An easy demonstration is given by considering Haar wavelets and their  random coefficients, which are not identically distributed. Therefore, the use of stochastic processes and random fields,  as suggested by Markkanen et al.\ \cite{MRHL15} and 
Hosseini \cite{BH17},  is attractive.  

Among other priors defined by stochastic processes and random fields, we 
mention hierarchical Mat\'ern type fields \cite{CIRL17}, which are approximated through 
stochastic partial differential equations.  

 Moreover, for consistent approximations of the posterior  in TV metric  or Hellinger distance, one typically needs  a stronger topology on the space of unknowns in order to obtain the needed uniform convergence (see Example 2 in \cite{HN17}).  This raises the question, do the approximations of the posterior converge in weaker topology on a more natural sample space?

  A common and popular choice of priors traditionally have been priors of a Gaussian form \cite{VIB98,KVZ11,RHL14}.  This is due to both attractive properties associated with the form, but also its applicability for computational purposes, such as in the context of partial differential equation (PDEs) based inverse problems \cite{BGMS13,CIRL17,PMSG14,RHL14}. With many PDEs the physical properties of the unknown such as a diffusion coefficient omit a heterogeneous form, so modelling them through a Gaussian prior is a natural choice. Such PDEs include impedance tomography, Darcy flow and the Navier-Stokes equation. However there are scenarios in which the unknown of interest {does not omit a representation typical of a Gaussian prior}, but instead has certain discontinuities and edges.  As a consequence a Gaussian prior can result in a poor approximation.

\subsection{Our Contribution}
The purpose of this paper is to study a new class of non-Gaussian priors which are motivated from $\alpha$-stable sheets \cite{CLL08,HL06,KSS13,TS94}. These priors are based on $\alpha$-stable random variables. They can be expressed through the general form
\begin{equation*}
 {X} \sim S_{\alpha}(\sigma,\beta,{\gamma}),
\end{equation*}
where $\alpha \in (0,2]$ denotes stability parameter, $\beta\in [-1,1]$ denotes the skewness parameter, ${\gamma} \in \mathbb{R}$ represents the shift parameter, and $\sigma \in \mathbb{R}$ which describes the location. These processes are of particular interest as they incorporate both smooth and rough features through various distributions, such as Gaussian $(\alpha=2)$ and Cauchy $(\alpha=1)$ \cite{CIRL17,RHL14,TJS16}. Both these processes have been used for Bayesian inversion where a comparison was conducted in the work of Markkanen et al.\ \cite{MRHL15}. Much of the current literature has focused on developing computational methodologies for non-Gaussian priors. In contrast to this we aim to build a theoretical understanding of $\alpha$-stable sheets, which are highlighted below as our main contributions.
\begin{itemize}
\item[(i)] We aim to ensure that both  well-definedness and well-posedness  of the inverse problem, when the prior is specified as a stable random field, is attained. This can be achieved through a number of assumptions which we place on the Radon-Nikodym derivative, which is the form the posterior distribution takes. Depending on the assumptions we verify well-posedness in weak topology and  in total variation metric. Our discussion extends the results of Hosseini et al.\ \cite{BH17,HN17} and Sullivan
\cite{TJS16} towards conditions that do not explicitly involve 
slowly growing bounds for the forward mapping. However, we emphasize that the cases in \cite{BH17, HN17, TJS16} are valid for stable random fields.

\item[(ii)] Under certain conditions and assumptions, we provide numerous convergences results. To study convergence we introduce different representations of the sheets, which includes a series expansion and in integral form. Specifically for this work we will focus on both $L^p$ convergence and convergence in a fractional Sobolev spaces $H^s_p$. This will be examined for different values the stability parameter $\alpha \in (0,2]$ can take. We analyze the sample path regularity of these processes to ensure convergence, as these processes omit discontinuities. 
\item[(iii)] A discretization of the $\alpha$-stable sheets is established, which is constructed through hypercubes and discretized random walks. The motivation of the particular discretization is taken from \cite{MRHL15}, where these sheets will act as a finite-dimensional approximation. This will lead onto providing various convergence results. However unlike the different results discussed above, we will solely concentrate on $L^p$ convergence.
\end{itemize}

\subsubsection{Outline.}
{The layout of this work is split in the following manner. In Section 2 we provide an overview of preliminary material and notation which is required for the rest of paper. This includes a discussion on Bayesian inverse problems and their well-posedness. This will lead onto Section 3 where we discuss $\alpha$-stable stable sheets  and the forms they take.  Then we carefully re-evaluate the existing results on sample spaces of  $\alpha$-stable random fields and update them to suit the needs of Bayesian inverse problems.    We extend these results to the case of the posterior in Section 4 where we describe  parametric approximations  of the sheets and provide their posterior convergence.  Finally in Section 5 we conclude with some final remarks while mentioning further areas of research.}
\section{Background material}
\subsection{Notation and preliminaries}

Let $(\Omega,\Sigma,P)$ represent a complete probability space, with sample space $\Omega$,  $\sigma$-algebra $\Sigma$ and probability measure ${P}:\Sigma \rightarrow [0,1]$. The set of all real-valued  random variables on $\Omega$ is denoted with $L^0(\Omega)$.  The conditional expectation $ \mathbb{E}[f|\Sigma_0]$  of a function $f \in L^1(\Omega,\Sigma,{P})$, given a $\sigma$-algebra $\Sigma_0 \subset \Sigma$ is a $\Sigma_0$-measurable function, where
\begin{equation*}
\int_{\mathcal{A}} f d{P} = \int_{\mathcal{A}} \mathbb{E}[f|\Sigma_0]d{P}, \text{ for all } \mathcal{A} \in \Sigma_0.
\end{equation*}

Let $F$ be a separable Banach space equipped with its Borel $\sigma$-algebra $\mathcal F$.
 We denote with $F'$ the dual space of $F$ i.e.\ the space of all continuous linear 
 forms $\lambda:F\rightarrow \mathbb R$. We denote the action of $\lambda $ on $u\in F$ with 
 duality  $\lambda(u)=\langle u, \lambda\rangle _{F,F'}$ between $F$ and $F'$ and equip
 the linear space  $F'$ with the strong topology  $\Vert \lambda\Vert_{F'} =
  \sup_{\Vert u\Vert \leq 1} \langle u,\lambda\rangle_{F,F'}$. 
 
We say that $U:\Omega\rightarrow F$ is $F$-valued random variable, if  $U^{-1}(B)\in \Sigma$ for 
all Borel sets $B\in\mathcal F$.  
 \begin{remark}
 \label{remark:pp}
For a separable Banach space $F$,    weakly measurable mappings $U:\Omega\rightarrow B$ are measurable,  since the Borel $\sigma$-algebra $\mathcal F$ is generated by cylinder sets of the form 
$$
\{ u\in F : (\langle u,\lambda_1\rangle, \dots, \langle u,\lambda_m\rangle)\in B\}, 
$$
where $m\in \mathbb N$, $B\subset \mathcal B(\mathbb R^m)$,  and $\lambda_k\in F'$. For separable duals, we may choose
$\{\lambda_k\}_{k=1}^\infty$  to be any  countable dense set  of the dual space of $F$ (see Theorem 6.8.9 in \cite{VIB07}).  
 \end{remark}

The distribution $P\circ U^{-1}$   of  $U  $ on $(F,\mathcal F)$  is denoted with $\mu_U$.  A conditional distribution of $U$ given another Banach-space valued random 
variable $Y:\Omega \rightarrow G$,  is defined as $\mu_{U}^{Y(\omega)} (B):= \mathbb E [1_B(X)| Y^{-1}(\mathcal G)](\omega)$ for all Borel sets $B$ of $F$, where $\mathcal G$ is the Borel $\sigma$-algebra of the separable Banach space $G$.  The notation $\mu_U^{Y(\omega)}$ 
emphasizes the fact that conditional expectation given $\sigma$-algebra generated by another 
random variable $Y$ can be expressed as a function of $Y$.  In separable Banach spaces, 
conditional distributions have  regular  version in the sense that $y\mapsto \mu_U^y(B) $ is measurable from $G$ to $[0,1]$ for every $B\in \mathcal F$ and 
$\mu_U^y$ is a probability measure on $(F,\mathcal F)$ for every  $y\in G$.  We remark, that 
only the joint distribution of $U$ and $Y$ is needed to  determine the distributions 
$\mu_U^y$ up to a $\mu_Y$-null set (see Theorem 2.4 in \cite{SL12i}). 

Two separable  Banach space-valued random variables $U:\Omega\rightarrow F$ and $\eta:\Omega\rightarrow \widetilde G$  are called independent, if    $1_B(U)$ and $1_{D}(\eta)$ are independent for any Borel sets $B\subset F $ and $D\subset \widetilde G $. For independent Banach-space valued random variables $U,\eta$ and continuous  
function  $K:F\times \widetilde G\rightarrow G$, the conditional distribution of 
$K(U,\eta)$ given $U=u$ is the distribution of  $K(u,\eta)$ (see Lemma 3.2 in \cite{SL12i}).

A characteristic function of a probability measure  $\mu_U$ on  $(F,\mathcal F)$ is 
a mapping $\phi: F'\rightarrow  \mathbb C$ given by 
\begin{equation*}
 \phi(\lambda) = \int \exp\left( i \langle u,\lambda\rangle _{F,F'} \right) \mu_U(d u).
\end{equation*}

For further details on measures and regular conditional distributions on Banach spaces,  we refer to 
\cite{VIB98,VIB07}.

   An $F$-valued random variable  $U$  has a stable distribution if, for any positive constants $a$ and $b$, there exists  positive constants $c$ and $d$  such that
\begin{equation}\label{eq:stable_def}
aU_1+bU_2 = cU + d,
\end{equation}
in distribution, where $U_1$ and $U_2$ are independent copies of $U$.  Real-valued stable random variables $U$ have characteristic functions
\begin{equation*}
\phi(t)=\mathbb E[ \exp\left( i t U \right) ] =
\begin{cases} 
\exp\left( -\sigma^\alpha |t|^\alpha \left(1- i\beta \operatorname{sgn}(t) \right)\tan\left( \frac{\pi \alpha}{2} \right) + i \gamma t\ \right) \text{ when } \alpha\not= 1
 \\
 \exp\left( -\sigma^\alpha |t|^\alpha \left( 1+i\beta \frac{2}{\pi} \operatorname{sgn}(t) \ln (|t|)\right) + i \gamma t\ \right) \text{ when } \alpha\not= 1,
\end{cases}
\end{equation*}
where  parameters $\alpha \in (0,2]$,   $\beta \in [-1,1]$, $\sigma \in \mathbb{R}$ and $\gamma \in \mathbb{R}$. Parameters $\alpha$ and $\beta$ are referred as the stability and
skewness parameter,  respectively. The distribution of a real-valued stable random variable 
is denoted with 
\begin{equation*}
 U \sim S_{\alpha}(\sigma,\beta,\gamma).
\end{equation*}
The  distribution of real-valued stable random variable  is called symmetric if $\beta=\gamma=0$.  
When $U$ is $F$-valued stable random variable, we 
call  it symmetric,  if the composition of $U$ with any
continuous linear form $\lambda\in F'$ is 
symmetric. Especially, we can identify the distribution of  symmetric $U$ 
through its characteristic function
$$
\mathbb E \left[ \exp (i \langle U,\lambda\rangle _{F,F'}) \right]  = 
\exp\left( -\sigma_\lambda ^\alpha\right).
$$
{For the purpose of this work we will focus on a special type of symmetric stable random variables, $\alpha$-stable sheets, which 
we  will use as heavy-tailed priors in Bayesian inverse problems for $L^p$-valued unknowns. This will be introduced and discussed in detail in Section \ref{sec:sheets}}.

\subsection{Bayesian inverse problems}
\label{ssec:BIP}

We will first recall the basics of Bayesian inverse problems in infinite-dimensional spaces. Let $F$, $G$ and  $\widetilde G$ denote separable  Banach spaces equipped with their Borel $\sigma$-algebras $\mathcal  F, \mathcal G$ and $\widetilde {\mathcal  G} $, respectively. An inverse problem is concerned with the recovery of some quantity of interest $U \in F$ from data  $Y \in G$, where we will consider noisy models of the type
\begin{equation}
\label{eq:inv}
Y = K(U,\eta),
\end{equation}
such that $\eta$ is $\widetilde G$-valued random noise and $K:F\times \widetilde G \rightarrow G$ is a  continuous mapping. We  take the Bayesian approach and  model $U$ as an $F$-valued random variable with prior distribution $\mu$ on  $\mathcal F$.  A common setup is to take 
$K(u,y)= L(u)+y$, where $ L: F\rightarrow G$ is a continuous mapping.

 When distributions of $U$ and $Y$ have  probability density functions   (e.g.\  when $F$ and $G$
  are finite-dimensional), we  write the familiar Bayes' formula 
for the posterior distribution of $U$ given $Y$ 
as

\begin{equation}
\label{eqn:Bayesian}
\begin{split} 
f_U(u|Y=y) =& \frac{f_Y(y|U=u)f_{U}(u)}{f_Y(y)}  \\
 \propto & f_Y(y|U=u)f_{U}(u), 
\end{split}
\end{equation}
where  $f_Y(y|U=u)$ is the likelihood and  $f_{U}(u)$ is  the prior probability density of  $U$, which   represents our initial beliefs about the unknown $U$.  We will denote with 
$\mu$ the prior distribution of $U$ and with $\mu^y$ the posterior distribution of $U$ on $F$. The non-existence of Lebesgue's measure in infinite-dimensional setting prohibits us from using \eqref{eqn:Bayesian} in the infinite-dimensional case.   Instead, one uses  different measures in place of the Lebesgue's measure. Let us recall the basics of infinite-dimensional Bayesian inference by 
considering some formal candidates for  replacements  of  Lebesgue's measure in \eqref{eqn:Bayesian}, which lead to the well-known representation of the 
posterior distributions  (see \cite{AMS10}).

Avoiding the use of the  posterior density in  \eqref{eqn:Bayesian} is straightforward,  where we just take
 integrals and consider the posterior distribution instead of posterior density. Similarly,  the  prior density can be avoided by integrating with respect to prior  distribution $\mu(du)$ instead of   $f_U(u) du$, where $du$ is the  Lebesgue's measure.    From 
 \eqref{eqn:Bayesian}, we formally derive the  posterior distribution 
 $$
 \mu^y(B)= \frac{ \int_B f_Y(y|U=u) \mu(du)}{\int_F f_Y(y|U=u) \mu(du)}. 
 $$
 Expressing the integral  with  respect to prior measure $ \mu(du) $  and considering  posterior distributions instead of  posterior  densities  handles   two out of three  problematic densities.

The critical part of  generalizing \eqref{eqn:Bayesian} to infinite dimensions  is finding a generalization  for the  likelihood function $f_Y(y|U=u)$, which has turned  out to be nontrivial and sometimes even  impossible (see Remark 4 and 5,  together with  a simple Gaussian counterexample,  Remark 9, in  \cite{SL12i}).

If the conditional probability distribution of $Y$ given $U=u$, which we denote with 
$\mu_Y^u$,  has Radon-Nikodym densities  
$
\frac{d\mu_Y^u }{d\nu}(y)
$  
with respect to  a  common $\sigma$-finite measure $\nu$ on $G$ for ($\mu$-almost) all $u\in  F$, the  Bayes' formula continuous to hold in the sense that the posterior distribution  has Radon-Nikodym density
 (see \ \cite{kallianpur1968,SL12i})
\begin{equation}\label{eqn:inv_inf}
\frac{d\mu^y}{d\mu} (u) \propto  \frac{d\mu_{Y}^u}{d\nu}(y) =: \exp(-\Phi(u,y)),
\end{equation} 
 with respect to the prior distribution $\mu$. In \eqref{eqn:inv_inf}, the mapping 
 $\Phi: F\times G \rightarrow \overline{ \mathbb R}$ is an extended real-valued mapping, which 
 can be chosen to be  jointly $\mu\otimes \nu$- measurable on $F\times G$  (see Theorem 2 in  \cite{kallianpur1968}).
 
\begin{defn}\label{def:NDLL}
We say that  the inverse problem \eqref{eq:inv} is dominated, if  there exists a   $\sigma$-finite measure $\nu$  so that the Radon-Nikodym densities 
\begin{equation*} 
\frac{d\mu_Y^u }{d\nu}(y) =\exp(-\Phi(u,y)),
\end{equation*}
define  a jointly measurable mapping   $\Phi: F\times G \rightarrow \overline{ \mathbb R}$. In this work, we call $\Phi$ negative dominated loglikelihood (NDLL).
\end{defn}

In this work  we will focus on the basic case, where $G$ is finite-dimensional and the generalized likelihood
$\exp(-\Phi(u,y))$ is bounded.  For example,  when the  $U$ and $\varepsilon$ are statistically independent, 
we take $\Phi(u,y)= -\log f_{\varepsilon+L(u)}( y)$ for the observation $Y=L(U)+\varepsilon$.

\subsection{Well-posedness of Bayesian inverse problems}

The well-posedness of the posterior distribution, established first by Stuart \cite{AMS10} in the Gaussian nonlinear case,  means essentially that
the posterior distribution  $\mu^y $  depends  continuously on $y$ with respect to suitable topology on the space of probability distributions.   Well-definedness of the posterior is well-known for almost every observation (see \ \cite{kallianpur1968,SL12i} and references therein), but the characterization of sets of well-definedness has gained more popularity during the last decade.  We  recall sufficient conditions for  well-posedness of  the posterior distribution of  stable random sheets  $U$  in dominated inverse problems, with respect to weak topology and the total variation metric.  We follow the general scheme introduced in \cite{AMS10} and refined in \cite{TJS16}. 

Let the posterior distribution be of the form 
\begin{equation*}
\mu^y \propto    \exp \left( -\Phi  (x,y)\right)  \mu(dx),
\end{equation*}
where  $\Phi$  is NDLL as in Definition \ref{def:NDLL} and  the prior $\mu$ is the distribution of  the stable random sheet $U$.

In the next definition,  Conditions {WD1} and {WD2} are connected to well-definedness in the fully infinite-dimensional case.  Condition WP1 and WP2 are connected to well-posedness.   Condition {WP2} is connected to well-posedness in  weak topology  and  Condition {WP3} strengthens well-posedness so that it holds also  in  total variation metric.  Conditions
WD1 and WP1  intentionally leave  the finer sufficient properties of the NDLL undetailed, since our intention is to use a pure skeleton of high-impact assumptions.  The Condition PC1, which we will use later,  is connected to posterior convergence of the Bayesian inverse problem.

\begin{defn}
 We define the following conditions for a  NDLL   $\Phi:F\times G\rightarrow \mathbb R$ and 
 a distribution $\mu$ on $F$. 
\begin{enumerate}
\item [] \emph{\textbf{WD1.}}  There exists a  set $D_\mu \subset G$ such that the function $\exp(-\Phi(u,y))$ is $\mu$-integrable for every  $y\in D_\mu$.
\item [] \emph{\textbf{WD2.}}  The NDLL $\Phi$ is bounded on bounded subsets of $F\times G$.
\item[]  \emph{\textbf{WP1.}}  There exists a  set  $D_\mu \subset G$  such that the  function $y\mapsto  \int \Phi(u,y)\mu(du) $ is continuous on $D_\mu$ (in the relative topology). 

\item[] \emph{\textbf{WP2.}} The function $y\mapsto \Phi(u,y)$ is continuous on $G$ for every 
$u\in F$.
\item[]\emph{\textbf{WP3.}} The functions $y\mapsto \Phi(u,y)$  on $G$  are uniformly continuous  with respect to  $u\in B$ for any bounded subset $B$ of $F$. 

\item[]\emph{\textbf{PC1.}} The functions $u\mapsto \Phi(u,y)$  on $F$  are  continuous   for 
any $y\in G$.
\end{enumerate}
\end{defn}
In WP1, we  can replace convergence in  relative topology with the original topology whenever the set  $D_\mu$  is open. 

\begin{defn}
We say that the posterior distribution $\mu^y\propto  \exp(-\Phi(u,y)) \mu(du) $  is well-defined  on the set $D_{\mu}$,  if  the normalizing constant $Z_y= \int_F \exp (-\Phi(u,y)) \mu(du)$ is positive and finite for every  $y\in D_\mu$. 
 \end{defn}

 We recall the well-known fact that the posterior distributions are always almost surely well-defined. However,  the set $D_\mu$ is  not explicitly characterized. 
 The second part of the claim, which originates from Stuart \cite{AMS10}, is used when we wish to evaluate well-definedness on a given set $D_\mu$.  
\begin{thm} \label{thm:WD}
 Let $\Phi$ be a NDLL. 
\begin{itemize}
\item[(i)] There exists a set $D_\mu$ of full $\mu_Y$-measure such that 
the  posterior distribution  $\mu^y \propto\exp (-\Phi(u,y)) \mu(du)$ is  well-defined
 on the set $D_\mu$. 

\item[(ii) ] If  NDLL $\Phi$ and a prior distribution $\mu$  satisfy  Conditions  WD1 and WD2 
for a given set $D_\mu\subset G$,  then the posterior distribution  $\mu^y \propto\exp (-\Phi(u,y)) \mu(du)$ is  well-defined  on the set $D_\mu$. 
 \end{itemize} 
 \end{thm}
\begin{proof}
For the claim (i),  we write norming constants with the help of 
the Radon-Nikodym density 
\begin{equation*}
 \int _ F  \exp (-\Phi(u,y))  \mu(du)  = \int _ F  \frac{d\mu_Y^u }{d\nu}(y) \mu(du)
\end{equation*}
and integrate with respect to $\nu$ over $G$. Since $\mu_Y$ is a finite measure, the norming 
constants are finite $\nu$-almost everywhere. We choose $\{y: \int _ F  \exp (-\Phi(u,y))  \mu(du)<\infty\}$ to be $D_\mu$. By replacing $G$ with $D_\mu$ in the integral, we see that 
the claim holds  also $\mu_Y$-almost surely.  

We now show  (ii). Since all probability measures on separable Banach spaces are Radon, there 
exists a compact set $K\subset F$ such that $\mu(K^C)< \frac{1}{2}$. 
By Condition WD2, the NDLL  $\Phi$ is bounded on $K\times\{y\}$ and there is a 
constant $C>0$ such that $-C \leq \Phi(u,y)   \leq  C$ on $K\times\{y\}$. Therefore,
 the normalizing constant has a lower bound
\begin{equation*}
\int _ F  \exp (-\Phi(u,y))  \mu(du) \geq \int _K \exp (- C) \mu(du) \geq \exp(-C)/2,
\end{equation*}
for every  $y\in G$. By Condition WD1, the normalizing constant is also bounded for every $y\in D_\mu$.
\end{proof}

We will now turn to the well-posedness of  posterior distributions. We start with weak topology of measures, which does not require 
integrability of moments from the prior distribution. This is very attractive for stable distributions, which are heavy-tailed. 
\begin{defn} 
We say that  the posterior distribution $\mu^y \propto \exp(-\Phi(u,y)) \mu(du)$  is 
well-posed on the set $D_\mu\subset G$  in  weak topology,  if
$\mu^y$ is well-defined on $D_\mu$ and,  for every bounded continuous function 
$f:F\rightarrow \mathbb R$,  the equation 
\begin{equation*}
\lim_{k\rightarrow \infty} \mu^{y_k}(f)= \mu^{y}(f),
\end{equation*}
holds  whenever  $\lim _{k\rightarrow \infty} y_k =y$,  where $y_k,y\in D_\mu$, $k\in\mathbb N$.  
\end{defn}

 For well-posedness in the weak topology, it is enough to show  continuity of the norming constants, if $\Phi$ is suitably continuous. 
 \begin{thm}
Let  a  NDLL $\Phi$ and  a prior  distribution $\mu$ satisfy  Conditions  WD1, WD2, WP1, and WP2. Then the  posterior distribution $\mu_y\sim \exp(-\Phi(u,y))\mu(du)$ is well-posed on the set  $D_\mu$ in weak topology.
\end{thm}
\begin{proof}
 
By Condition WP1 and  WP2 
\begin{equation*}
 Z_{y} \exp(-\Phi(u,y)) =  \lim_{k\rightarrow \infty} Z_{y_k} \exp(-\Phi(u,y_k)) = 
  \operatorname{liminf}_{k\rightarrow \infty} Z_{y_k} \exp(-\Phi(u,y_k)),
\end{equation*}
which we integrate over an  open set $U\subset F$. By Fatou's lemma
\begin{equation*}
 \int_U   Z_{y} \exp(-\Phi(u,y)) \mu(du)   \geq \operatorname{liminf}_{k\rightarrow \infty}  \int_U   \frac{1}{Z_{y_k}}\exp(-\Phi(u,y_k)) \mu(du).
\end{equation*}
Since the open set $U$ can be chosen freely, we arrive at a well-known equivalent criteria for weak convergence of distributions.
\end{proof}

Next, we recall a stronger mode of convergence from
\cite{DHS12,BH17,HN17,AMS10,TJS16}.
\begin{defn} 
We say that  the posterior distribution $\mu^y\propto \exp(-\Phi(u,y)) \mu(du)$ is well-posed on  the set $D_\mu\subset G$ in  total variation metric,  if
$\mu^y$ is well-defined on $D_\mu$ and 
\begin{equation*}
\lim_{k\rightarrow \infty}  \sup_{A\in \mathcal B(F)} |\mu^{y_k}(A) -  \mu^{y}(A)|,
\end{equation*}
 whenever  $\lim _{k\rightarrow \infty} y_k =y$,  where $y_k,y\in D_\mu$, $k\in\mathbb N$.  
\end{defn}

\begin{remark}\label{rem:tight}
For a  Banach space $F$, the uniform tightness of the family of distributions $\mu_k$ on $F$ is equivalent to  the condition that for every $\epsilon>0$ and every $r>0$, there exists a finite number of open 
balls $B_{r,\epsilon}$ of $F$ such that 
\begin{equation*}
\mu_k\left (F\backslash \cup B_{r,\epsilon}\right)< \epsilon,
\end{equation*} 
for every $k$ (See Remark 2.3.1 in \cite{VIB18}).
\end{remark}

Next, we study well-posedness of the posterior distribution  in total variation metric. Again, the convergence of norming constants is imperative. The milder assumptions than in 
\cite{DHS12,BH17,HN17,AMS10,TJS16} are achieved by relaxing the  posterior Lipschitz continuity to ordinary continuity.
\begin{thm}\label{thm:TV}
Let  a   NDLL $\Phi$ and  a prior distribution $\mu$  satisfy Conditions WD1, WD2, WP1, and  WP3. 
Then  the  posterior distribution $\mu^y \propto \exp (-\Phi(u,y)) \mu(du)$ is
well-posed on the open set $D_\mu$   in total variation metric.
\end{thm}
\begin{proof}
  Let  $y_k\rightarrow y$, where  all $y_k,y\in D_\mu$. Then
  all $y_k,y$ belong to a bounded set $B$. 
     
We  use the equivalent definition of uniform tightness in  Remark \ref{rem:tight}.   Let $\epsilon>0$ and $r>0$.  By tightness of $\exp(-\Phi(u,y))\mu(du)$,   there exists a finite number of balls $B_{r,\epsilon/2}^p$ such that 
$$
\int_{F\backslash\cup_p  B_{r,\epsilon/2}^p } \exp(-\Phi(u,y))\mu(du)<\epsilon/2. 
$$
By Conditions WP3, the NDLL $\Phi$ has an upper  bound 
\begin{equation}\label{eqn:WP_1}
\begin{split}
\Phi(u,y_k)= &  \Phi (u,y) + \left( \Phi (u,y_k)- \Phi(u,y)\right)  \\
\leq &\Phi(u,y)   + \sup_{u\in B'} \left|  \Phi (u,y_k)- \Phi(u,y) \right|\\
  \leq & \Phi(u,y) + R\varepsilon,
  \end{split}
\end{equation}
for all $k>N=N_{R,\epsilon,r}$ and  $u$  from a bounded subset $B'$ of $F$, which we take to be 
  the finite union $\cup_p B_{\epsilon/2,r}^p$.  We estimate 
\begin{equation*}
\begin{split}
\int _{\cup_p  B_{r,\epsilon}^p}  \exp(-\Phi(u,y_k)) \mu(dx) 
\geq &  \int _{\cup_p  B_{r,\epsilon}^p}  \exp(-\Phi(u,y) - R\epsilon) \mu(dx)\\
>&  (Z_y-\epsilon/2) \exp(-R\epsilon).
\end{split} \end{equation*}
We choose $R$ so that 
$$
 (Z_y -\epsilon/2) \exp(-R\epsilon ) \geq  Z_y-3\epsilon/4,
$$
that is, 
$$
R  \leq \frac{1}{\epsilon } \ln \left( \frac{Z_y-\epsilon/2}{Z_y- 3\epsilon/4}\right) .
$$
Then 
\begin{equation*}
\int_{F\backslash\cup_p  B_{r,\epsilon/2}^p } \exp(-\Phi(u,y_k))\mu(du) <
 3\epsilon/4 + Z_{y_k}-Z_y.
\end{equation*}
Since the normalizing constants converge by Condition WP1, we may choose $N'>N$ so that 
$|Z_{y_k}-Z_y|< \epsilon/4$ when $k>N'$. 

For $k\leq N'$,  there exists finite number of open balls 
$
B_{r,\epsilon,k}^{p_k }$ such that 
$$\int_ {F\backslash \cup_{p_k } B_{r,\epsilon,k}^{p_k}} \exp(-\Phi(u,y_k)) \mu(du) <\epsilon.$$  
The finite collection of balls $B_{r,\epsilon/2}^p$ together with  the  finite number of balls  $B_{r,\epsilon,k}^{p_k}$, where $k\leq N'$,  fulfills the required  condition in Remark 
\ref{rem:tight}.  Hence, $\Phi(u,y_k)\mu(du)$ are 
uniformly tight.  The uniformly  tight family of measures $ \Phi(u,y_k)\mu(du)$ is also bounded. 
Indeed, the converging sequence $y_k$ belongs to a bounded set $B$ and  by uniform tightness, there exists a compact set $K'=K'_\epsilon$ so that  
\begin{equation*}
\int_{K'\cup F\backslash K'} \exp(-\Phi(u,y_k)) \mu(du)   \leq 
\int   \sup_{(u,y)\in K'\times B}  \exp ( -\Phi(u,y)) \mu(du)+  \epsilon,
 \end{equation*}  
 where we  apply Condition WD2. 
 
 Finally, we verify the convergence of  $ \exp(-\Phi(u,y))\mu(du)$, which 
 follows directly from tightness and Equation \eqref{eqn:WP_1}. Indeed,
 \begin{equation*}
 \begin{split}
  &\int \left |  \exp(-\Phi(u,y_k)) - \exp (-\Phi(u,y)) \right|   \mu(du)  \\
 &\leq  \int _{K'}    \exp (-\Phi(u,y))  
 \left |  (\exp(\Phi(u,y) -\Phi(u,y_k) ) - 1 \right| \mu(u)  +  2  \epsilon 
  \\  & \leq Z_y \sup_{u\in K'}  \left| \exp(\Phi(u,y) -\Phi(u,y_k) ) - 1 \right|  + 2\epsilon, 
 \end{split}
 \end{equation*} 
 where we apply Condition WP3. 
\end{proof}

We do not discuss well-posedness in Hellinger topology, which 
is covered in \cite{BH17,HN17,TJS16}, since we have not been able to improve these results. The main obstacle is that  conditions characterized by moments arise more naturally for Hellinger distance.  

\section{$\alpha$-stable random measures and sheets} \label{sec:sheets}
In this section we review $\alpha$-stable fields, which will later serve as priors $U$. We highlight certain properties  and assumptions of $\alpha$-stable random fields  that are required in order to analyze the convergence. As our discretization scheme for the unknown 
function $U$  is  based on finite-difference approximations on certain function spaces, we need to verify that $\alpha$-stable random sheets have enough regularity to carry out the convergence analysis.  We aim to understand the convergence both in terms of probability and functional analysis.

To begin we require the concept of $\alpha$-stable stochastic integrals, which we 
recall from \cite{TS94}. Consider  measure space $(E,\Sigma(E),\overline{m})$ where 
\begin{equation*}
\Sigma(E)_0 = \{A \in \Sigma(E): \overline{m}(A) < \infty \},
\end{equation*}
denotes a subset of $\Sigma(E)$ that consists of  sets of finite $\overline{m}$-measure.
In our inverse problem, the unknown  $U$ will be a random field defined on $E\subset \mathbb R^d$.

\begin{defn}
Let $0<\alpha<2$. A random  $\sigma$-additive set function
\begin{equation*}
M:\Sigma(E)_0 \rightarrow L^0(\Omega),
\end{equation*}
is called an $\alpha$-stable random measure on $(E,\Sigma(E)_0)$ with control measure $\overline{m}$ and skewness parameter $\beta$, if it is independently scattered and for every 
$A\in\Sigma(E)_0$, 
\begin{equation*} 
M(A) \sim S_{\alpha} \Bigg((\overline{m}(A))^{1/ \alpha},\frac{\int_{A}\beta(x)\overline{m}(dx)}{\overline{m}(A)},0\Bigg),
\end{equation*}
 The random measure $M$ is called symmetric, if $\beta=0$.
\end{defn}
{The above definition holds for non-Gaussian examples, hence $\alpha \neq 2$}. By stating independently scattered we mean that if $A_1,A_2, \ldots , A_k$ belong to $\Sigma(E)_0$ and are disjoint then the random variables  $M(A_1),M(A_2),\ldots ,M(A_k)$ are independent. Furthermore, $\sigma$-additivity means that if $A_1,A_2,\ldots$, that belong to $\Sigma(E)_0$, are disjoint and $\cup^{\infty}_{j=1}A_j \in \Sigma(E)_0$ then 
\begin{equation*}
M\Big(\bigcup^{\infty}_{j=1}A_j\Big) = \sum^{\infty}_{j=1}M(A_j) \ \ \textrm{almost surely.}
\end{equation*}

 Stochastic integrals of deterministic functions $f$ with respect to    $\alpha$-stable random measure $M$ are defined similarly to the Gaussian case, through limits of  simple functions $f_k$. However, the convergence holds in a weaker sense. Namely 
 \begin{equation}\label{eqn:simplef}
 \int_{E} f(x) M(dx)=  \lim_{k\rightarrow \infty}\int_E f_k(x) M(dx),
 \end{equation}
 in probability if (and only if) $\lim_{k\rightarrow \infty} f_k=f$ in $L^\alpha(E)$ (Proposition 3.5.1 in \cite{TS94}).  The values of the random variable  $\int f(x) M(dx)$ can be  specified almost surely  by e.g.\ choosing a  subsequence $\int f_{k_j}(x)M(dx)$ that converges 
 almost surely to $\int _E f(x) M(dx)$.  Recall, that $L^\alpha$ space is only a complete metric space, not a normed space, when  $0<\alpha<1$. 
    
The distribution of the stochastic integral $\int f (x) M(dx)$ is 
$S_{\alpha}(\sigma_f,\beta_f,\gamma_f)$, where
\begin{align*}
\sigma_f &= \Big(\int_{E} |f(x)|^{\alpha}\overline{m}(dx)\Big)^{1/\alpha}, \\
\beta_f &= \frac{\int_{E} f(x)^{\langle \alpha \rangle}\beta(x)\overline{m}(dx)}{\int_{E}|f(x)|^{\alpha}\overline{m}(dx)}, \quad {y^{\langle \alpha \rangle} := |y|^{\alpha} \mathrm{sign}(y)}, \\
\gamma_f &= 
\begin{cases}
   0 & \text{if } \alpha \neq 1, \\
   -\frac{2}{\pi}\int_{E}f(x)\beta(x)\ln |f(x)|\overline{m}(dx)  & \text{if } \alpha = 1.
  \end{cases}
\end{align*}
We consider modeling our unknown  $U$ as an $\alpha$-stable sheet on the hypercube
  $[0,1]^d$. 
\begin{defn}
A random field $U$ on $[0,1]^d$  is called a symmetric $\alpha$-stable random sheet  if it   can be expressed (up to a version) as a  stochastic integral
\begin{equation}\label{eqn:sheet}
U(x)= \int_{[0,1]^d} f(x,x')M(dx'), \; x\in [0,1]^d,
\end{equation}
where 
\begin{equation}\label{eqn:sheet_f}
f(x,x')=\begin{cases}
1 \text{ when }  x_i'\leq x_i  \text{ for all }  i=1,\dots, d\\
0 \text{ otherwise},
\end{cases}
\end{equation}
and $M$ is  symmetric $\alpha$-stable random measure on $([0,1]^d,\mathcal B([0,1]^d), |\cdot|)$ with Lebesgue's measure $|\cdot |$ as the control measure.
\end{defn}
The  $\alpha$-stable random sheet has marginal distributions  
\begin{equation*}
U(x) \sim S_{\alpha} \Bigg((x_1\cdots x_d)^{1/ \alpha},0,0\Bigg), 
\end{equation*}
where $x=(x_1,\dots,x_d) \in[0,1]^d$. Moreover, the values $U(x)$ and $U(x')$ are  statistically dependent.

\subsection{Sample paths}
\label{ssec:cauchy}

Let us  now concentrate on the nature of the mapping $x\mapsto U(x;\omega)$ for fixed $\omega$,  where each $U(x)$ is defined  by \eqref{eqn:sheet}.  This is an important point, because we are interested in modelling our unknown function with  $U$ and wish to specify a Banach space $F$  where $U$ lives. In other words, we wish to describe $U$ as an $F$-valued random variable for some Banach space $F$. At this point,  we have defined $U$ as a random field,  through a family of random variables. 
We will heavily utilise another way of describing stable random fields, the  so-called LePage series representation   (\cite{TS94}, Theorem 3.9.1), which is often used in deriving sample path properties of stable random fields.

For the convenience of the reader,  we provide the proofs below and begin with two 
preparatory lemmas. We recall, 
that  arrival times $\Gamma_k$ of a  Poisson process with arrival rate 1 can be expressed as 
$$
\Gamma_k =\sum_{j=1}^k \lambda_j,
$$
where $\lambda_j$ are independent identically distributed 
random variables with common probability density ${f(x)= \exp(- x)1_{[0,\infty)}(x)}$. 

\begin{lem}\label{lem:poisson}
Let $\Gamma_k$ be arrival times of a Poisson process with arrival rate 1. There exists 
 $c(\omega),C(\omega) >0$ and $K(\omega)\in \mathbb N$ so that 
 $ c(\omega) k \leq \Gamma_k(\omega)  \leq C(\omega) k $ for all $k\geq K(\omega)$    and  for $P$- almost every $\omega$. Moreover, the series  
$$
\sum_{k=1}^\infty \Gamma_k^{-\kappa},
 $$
converges almost surely for all $\kappa>1$.
\end{lem}
\begin{proof}
By the law of large numbers
\begin{equation*}
\lim_{k\rightarrow \infty }\frac{ \Gamma_k}{k}=\mathbb E[\lambda_j] =1,
\end{equation*}
almost surely. Therefore, $\Gamma_k \sim  k $ for large $k$ and there exists  $c=c(\omega), C=C(\omega)>0$ and  integers $K=K(\omega)$ such that 
$ c(\omega ) k \leq \Gamma_k (\omega) \leq C(\omega) k $ for all $k> K(\omega)$
 almost surely.  Inserting the lower bound in the series 
 \begin{equation*}
 \sum_{k=1}^\infty \Gamma_k (\omega) ^{-\kappa} \leq  \sum_{k=1} ^{K(\omega)} 
 \Gamma_k (\omega) ^{-\kappa} + 
     c(\omega) ^{-\kappa}\sum_{k=K(\omega)+1}^\infty   k^{-\kappa},
 \end{equation*}
 shows that the series converges almost surely.  
\end{proof}

\begin{lem}\label{lem:cond3}
Let $(a_k), (b_k)$  be mutually independent 
random sequencies, and let $g$ be a separable Banach space-valued  Borel  measurable function. The series 
$$
\sum_{k=1}^\infty g(a_k, b_k),
$$
converges almost surely in $F$ if and only if the series 
$$
\sum_{k=1}^\infty g(a_k ,b_k^0),
$$
converges for almost every sample $b_k^0$ of $b_k$.
\end{lem}

\begin{proof}
 If  $A\in \Sigma$ is any event,  say
 \begin{equation*}
A= \left \{\omega \in \Omega:  \Vert \sum_{k=1}^\infty g(a_k,b_k) \Vert <\infty \right\},
\end{equation*}
then
\begin{equation*}  
P(A) = \mathbb E \left[ \mathbb E[1_A | b_k , \ k=1,2,\dots ] \right]=1,
\end{equation*}
if and only if $\mathbb E[1_A | b_k , \ k=1,2,\dots]=1$  almost surely. Indeed,
 conditional expectation of  $1_A$ is at most 1. A simple proof by contradiction shows that 
 the conditional expectation must equal 1 almost surely if $P(A)=1$. The other direction is
 trivial. 

\end{proof}

The next theorem  provides  a  series representation for the  symmetric $\alpha$-stable random measure $M$ with Lebesgue's control measure. In the theorem, we prove series representations 
of stochastic integrals,  when $f\in L^{p}$.
This suffices for our purposes, because we can always choose $p>\alpha$ for the functions that we study. The approach helps us to use  almost surely equivalence of stochastic integrals  instead of the more  common concept of equivalence in distribution.

\begin{thm}\label{thm:rrep}
Let $0<\alpha<2$ and  $E\subset \mathbb R^d$ be a measurable set with  $0<|E|<\infty$.  Let $\Gamma_k$ be  arrivals times of a Poisson process with  arrival rate 1.  Let $(V_k, \rho_k)$ form an  i.i.d. sequence of random vectors  independent of $\Gamma_k$ that consist of 
 uniformly distributed $d$-dimensional random vectors  $V_k$ on $E$, and  $\{-1,1\}$-valued random variables $\rho_k$ whose  conditional distribution given $V_k$ is   $P(\rho_k =1 |V_k)= 1 - P (\rho_k =-1|V_i)= 0.5$. Let
$$
C_\alpha =\left(\int _0^\infty x^{-\alpha} \sin(x) dx \right)^{-1}.
$$
Then
\begin{equation}\label{eqn:seriesn}
M(A):=  (C_\alpha |E|) ^{1/\alpha}\sum _{k=1}^\infty \rho_k \Gamma_k^{-1/\alpha} 
 1_A(V_k), 
\end{equation} 
where $A\subset E$ are Borel sets, defines  a symmetric $\alpha$-stable random 
measure $M$ with Lebesgue's control measure on $E$. 

If  $f\in L^{p} (E)$,  where  $p\geq 1$ when $0<\alpha<1$ and 
$p>\alpha$ otherwise,    then the  series
\begin{equation}\label{eqn:fun_series}
(C_\alpha |E|) ^{1/\alpha}\sum _{k=1}^\infty \rho_k \Gamma_k^{-1/\alpha} 
 f(V_k),
\end{equation}
converges almost surely and  it  coincides   with the stochastic integral $\int f(x) M(dx)$. 
\end{thm}
 \begin{proof}
   First, we will verify  convergence of  \eqref{eqn:seriesn}. 

For $0<\alpha<1$, a direct application of  Lemma  \ref{lem:poisson}
shows the convergence.  For $1\leq \alpha <2$, we proceed as in \cite{TS94} by using  Kolmogorov's three series theorem, but apply  it    under conditioning with respect to    $\Gamma_k=\Gamma_k^0$ in the spirit of 
Lemma  \ref{lem:cond3}.

  For fixed $\lambda>0$, 
 \begin{equation} \label{eqn:Kol_1}
 \sum_{k=1}^\infty   P( |\rho_k (\Gamma_k^0)^{-1/\alpha} 1_A(V_k)| > \lambda) \leq  
 \sum_{k=1}^\infty  P(
 1_A(V_k)>    (ck) ^{1/\alpha} \lambda) <\infty, 
  \end{equation}
since    $\Gamma_k ^0\sim  k$ for large $k$ by Lemma \ref{lem:poisson}.
Moreover, the sum of expectations of \\ $Y_k:=\rho_k (\Gamma_k^0) ^{-1/\alpha} 1_A(V_k) 1_{  |\rho_k (\Gamma_k^0) ^{-1/\alpha} 1_A(V_k)| \leq \lambda }$, that is, 
\begin{equation} \label{eqn:Kol_2}
 \sum_{k=1}^\infty \mathbb E[ Y_k]   =\sum_{k=1}^\infty   \mathbb E [\rho_k]   \; \mathbb E [ (\Gamma_k^0) ^{-1/\alpha} 1_{  (\Gamma_k^0) ^{-1/\alpha} 1_A(V_k)  \leq \lambda }],
\end{equation}
vanishes due to the distribution of $\rho_k$, and the sum of variances 
\begin{equation}\label{eqn:Kol_3}
 \sum_{k=1}^\infty \mathbb E[Y_k^2] \leq  \sum_{k=1}^\infty  (\Gamma_k^0)^{-2/\alpha} \mathbb E [1_A(V_k)^2],
\end{equation}
is finite by Lemma \ref{lem:poisson}.  Hence,   $M(A)$ is a well-defined random variable. 

Secondly, we show that $M(A)$  has the right distribution.  We first remark, that 
\begin{equation} \label{eqn:fin}
\lim_{n\rightarrow \infty }  (C_\alpha |E|) ^{1/\alpha}\ \left( \frac{\Gamma_{n+1}}{n}\right)^\frac{1}{\alpha} 
 \sum _{k=1}^n \rho_k \Gamma_k^{-1/\alpha}
 1_A(V_k) =M(A),
\end{equation}
by the law of large numbers. Additionally the random vector with components \\
$\frac{\Gamma_k}{\Gamma_{n+1} }\sim \operatorname{Beta}(k,n+1-k)$, $k=1,\dots, n$,
 is distributed as the random vector  $(U^{(1)}, \dots, U^{(n)})$ whose components are  
independent uniformly distributed  random variables $U_k$ on $(0,1)$ ordered increasingly.
In the finite sum \eqref{eqn:fin}, we may also reorder the random variables
 without changing the distribution of the sum.  This leads to identification of 
 $M(A)$ as  the limit of 
 sums of   independent random variables
 \begin{equation*} 
M(A)=  \lim_{n\rightarrow \infty }  (C_\alpha |E|) ^{1/\alpha}\left( \frac{1}{n}\right)^\frac{1}{\alpha} 
\sum _{k=1}^n \rho_k  U_k^{-1/\alpha} 
 1_A(V_k),
\end{equation*}
 which implies that  $M(A)$ is necessarily  a  stable random variable.   We still need to 
 define the index $\alpha$ and parameters of the stable distribution. 
To do this, we identify the  domain of attraction of the common  distribution of $ (C_\alpha |T|) ^{1/\alpha}\ \gamma_k  U_k^{-1/\alpha}    1_A(V_k) $. Since the common distribution is symmetric, we study  
the tail behaviour 
\begin{equation*}
\begin{split}
P ( | (C_\alpha |E|) ^{1/\alpha} \rho_k  U_k^{-1/\alpha}    1_A(V_k)|  >\lambda ) &= 
  P( U_k <  C_\alpha |E|\lambda ^{-\alpha}   \cap  V_k \in A ) \\
  &= \frac{|A|}{|E|}  \int _0^{\lambda ^{-\alpha } C_\alpha |E| }  1_{(0,1)} (t)  dt  = 
 |A| C_\alpha  \lambda^{-\alpha},
 \end{split}
\end{equation*}
 which implies  \cite{Geluk2000} that distribution of  $M(A)$ is  $\alpha$-stable. 
 Through tail behaviour, we identify  the distribution as   $S(|A|^\frac{1}{\alpha},0,0)$. 

From the characteristic function  of $M(A)$ it is evident, that $M$ is independently scattered.
For disjoint sets $A_k$,  the  sum  $M(\cup_k A_k)$ converges clearly almost surely  and  the sum of independent 
random variables  $\sum_k M(A_k)$   converges through It\={o}-Nisio theorem \cite{IN68}  almost surely and their  limit   coincide  in probability.  By selecting almost surely converging subsequences, the limits coincide almost surely, which shows that $M$ is  countably additive.

 When $0<\alpha<1$, the right hand side  of  $\eqref{eqn:fun_series}$ 
converges absolutely  by    Lemma \ref{lem:poisson}, since its  the conditional expectation 
given the sequence  $\Gamma_k=\Gamma_k^0$ is finite.  By choosing almost surely converging  subsequences from  the simple function  approximations  $M(f_j^\pm)$  that   converge in probability to $M(f)$, we identify the stochastic integral 
\begin{equation*}
M(f)=\lim_{i\rightarrow \infty} M(f_{j_i}) = \lim_{i\rightarrow \infty} 
  (C_\alpha |E|) ^{1/\alpha}\sum _{k=1}^\infty \rho_k  \Gamma_k^{-1/\alpha} 
 f_{j_i} (V_k),
\end{equation*}
with the almost surely converging series representation by taking the limit inside the sum by Lebesgue's dominated convergence  theorem.
 
  When $1\leq \alpha< 2$,  we divide $f\in L^{p}(E)$ into positive and negative parts $f^+$ and $f^-$,  and consider, where necessary, their strictly increasing simple function 
 approximations $f^\pm_{j}$. The right hand side of \eqref{eqn:fun_series} converges almost surely by Kolmogorov's three series theorem similarly to  \eqref{eqn:Kol_1} and \eqref{eqn:Kol_2}, since 
\begin{equation*}
\sum_{k=1}^\infty  P(
 f^\pm(V_k)>    (ck) ^{1/\alpha} \lambda) \leq \frac{\int f^\pm (x)^{p} dx} 
  { (ck) ^{p/\alpha} \lambda^{p}} <\infty ,
\end{equation*}
by Markov's inequality and the assumptions.   The convergence of the sum of  variances holds, since 
\begin{equation*}
\begin{split}
& \sum_{k=1}^\infty \mathbb E[ (\Gamma_k^0) ^{-2/\alpha} f^\pm (V_k)^2 1_{  |(\Gamma_k^0) ^{-1/\alpha} f^\pm (V_k)| \leq  \lambda }]  \\ &=   \sum_{k=1}^\infty \mathbb E
 \left [    (\Gamma_k^0) ^{-2/\alpha}  \int 
 f^\pm (x)^2 1_{  f^\pm (x)  \leq \lambda  (\Gamma_k^0)^{1/\alpha} } dx\right]  \\
  &\leq       \sum_{k=1}^\infty  
   (\Gamma_k^0) ^{-2/\alpha}  (\lambda (\Gamma^0_k) ^{1/\alpha})^{(2-p)} \int   f^\pm(t)^{p } dt \\
    &\leq 
 \sum_{k=1}^\infty  
    (\Gamma_k^0) ^{-p/\alpha}  \lambda^{2-p } \int   f^\pm(x)^{p}dx <\infty ,
\end{split}
\end{equation*}
by the assumptions. Because $p>\alpha\geq 1$ and $|E|<\infty$, the function    $f$ belongs to $ L^\alpha(E)$ by H\"older's inequality.  Moreover, the series representation is distributed as \\
$S_\alpha(  \left( \int |f(x)|^\alpha dx\right)^\frac{1}{\alpha} ,0,0)$. We next verify that simple function  approximations  $M(f_j)$ in 
  \eqref{eqn:simplef}, which  converge in probability to $M(f)$, have limit 
\begin{equation*}
M(f)=  (C_\alpha |E|) ^{1/\alpha}\sum _{k=1}^\infty \rho_k  \Gamma_k^{-1/\alpha} 
 f (V_k).   
\end{equation*}
Indeed,   for $0<\beta <\alpha$, 
\begin{equation*}
\begin{split}
P\left( \left |\sum _{k=1}^\infty \rho_k  \Gamma_k^{-1/\alpha} 
 (f-f_j ) (V_k) \right| > \varepsilon\right) \leq & \frac{\mathbb E \left[  \left|\sum _{k=1}^\infty \rho_k  \Gamma_k^{-1/\alpha} 
 (f-f_j ) (V_k)\right|^\beta\right] }{ \varepsilon^\beta }  \\
  < &
  C \varepsilon^{-\beta}\left( \int |f(x) - f_j(x)| ^\alpha dx \right)^{\beta/\alpha} ,
  \end{split}
\end{equation*}
by Markov's inequality. By choosing an  almost surely converging subsequences, we obtain 
\eqref{eqn:fun_series}.
 \end{proof}

We can now represent the stable random sheet as
\begin{equation*}
U(x)= (C_\alpha)^{1/\alpha } \sum_{k=1}^\infty \rho _k 
\Gamma_k^{-1/\alpha} f(x,V_k),
\end{equation*}
where $f$ is defined as in \eqref{eqn:sheet_f}.  A description of the sample paths of $U$  is given in the next theorem, which  allows application of  this prior  in many inverse problems.
The theorem is proved    for  general $f(x,x')$. We consider 
Borel measurable functions $f$  instead of more general equivalence classes in  $L^p$  in 
order to to guarantee that compositions of $f$ with random variables are themselves  random variables.  
\begin{thm}\label{lem:LPpaths}
Let $0<\alpha<2$, $1\leq p<\infty$, and $p>\alpha$. Let  $E\subset \mathbb R$ be Borel measurable with $0<|E|<\infty$,    and  let $C_\alpha,\rho_k,\Gamma_k,$ and $V_k$ be defined as in Theorem \ref{thm:rrep}.   Let $T\subset \mathbb R^d$ be a measurable set with $|T|>0$.

If $f:T\times E \rightarrow \mathbb C $ is Borel measurable and 
$ \int_T \int _E  |f(x,x')|^p  dx' dx  <\infty$, 
 then the random series representation
 \begin{equation}\label{eqn:series_U}
 U(x)= (C_\alpha |E|)^{1/\alpha }  \sum_{k=1}^\infty \rho _k 
\Gamma_k^{-1/\alpha} f(x,V_k),
 \end{equation}
converges almost surely  in $L^p(T)$.   As a result the distribution of 
$U$ on $L^p(T)$ is symmetric and stable. 
\end{thm}

\begin{proof}
When $x'\mapsto f(x,x')$ is Borel measurable, the truncated series $U^N$
  is a well-defined random variable in $L^p(T)$  because  
$f(x,V_k)$ are $L^p(T)$-valued random variables. Indeed,   $L^p(T)$ is a separable Banach space,  so weak and strong  measurability coincide by Remark \ref{remark:pp}. It is enough to consider a countable dense  set $\{g_j\}_{j=1}^\infty $ of the dual space $L^q(T)$,  $q^{-1}+p^{-1}=1$.  
Then  $\langle U_N, g_i\rangle$ are clearly measurable.

We will show that series \eqref{eqn:series_U} converges in $L^p(T)$ by using
methods introduced in \cite{TS94}.
Assume first that $p=1$.    Then 
\begin{equation*}
 \int _T \left|\sum_{k=1}^\infty  \rho_k  (\Gamma_k)^{-1/\alpha }  f(x,V_k)  \right| dx 
 \leq  \sum_{k=1}^\infty   (\Gamma_k)^{-1/\alpha }  \Vert f(\cdot , V_k)\Vert _{L^1(T)},
\end{equation*}
which has finite expectation by Lemma \ref{lem:cond3}, since 
\begin{equation*}
\mathbb E \left[ \Vert f(\cdot , V_k)\Vert _{L^1(T)}\right]
 = \int_T \int_E |f(x,x')| dx' dx< \infty.
\end{equation*}

Consider next $p>1$. By Lemma \ref{lem:cond3},  the series \eqref{eqn:series_U} converges  almost surely in $L^p(T)$ 
if it  converges  under conditioning with $\Gamma_k,V_k$, $k\in\mathbb N$.
  
Let us fix  $\Gamma_k,V_k $ to  constant values $\Gamma_k^{0}, V_k^0$   by conditioning. 
 By independence, the  truncated series  becomes
\begin{equation*}
\sum_{k=1}^N   \rho_k (\Gamma_k^0)^{-1/\alpha} 
 f(x,V_k^0).
\end{equation*} 

By It\=o-Nisio theorem \cite{IN68}, we only need to prove convergence 
in probability, which will follow by Markov's inequality,  if we can show that
\begin{equation}\label{eqn:forT}
\lim_{M,N\rightarrow \infty }\mathbb E\left[ \left  \Vert \sum_{k=M}^N \rho_k (\Gamma_k^0)^{-{1}/{\alpha}}  f(\cdot ,V_k^0) \right \Vert_{L^p(T)}\right]=0,
\end{equation}
for $\mu_{\{\Gamma_k,V_k\}_{k=1}^\infty}$-almost every sequence $(\Gamma_k^0,V_k^0)$.

Take first $1<p\leq 2$ and $\alpha\geq 1$.
By Jensen's inequality, 
\begin{equation*}
\begin{split}
&\mathbb E\left[ \left\Vert \sum_{k=M}^N \rho_k (\Gamma_k^0)^{-{1}/{\alpha}} f(\cdot , V_k^0) \right \Vert_{L^p(T)}\right]^p  \leq 
\mathbb E\left[  \int \left| \sum_{k=M}^N \rho_k(\Gamma_k^0)^{-{1}/{\alpha}}f(x,V_k^0) \right|^p  dx \right] \\
&=    \int   \mathbb  E\left[     \left(  \sum_{k=M}^N  \rho_k (\Gamma_k^0)^{-{1}/{\alpha}} f(x,V_k^0)    \right)^p   \right]  dx  \\
&\leq  \int     \left( \mathbb  E\left[  \left(  \sum_{k=M}^N  \rho_k (\Gamma_k^0)^{-{1}/{\alpha}} f(x,V_k^0)    \right)^2 \right]\right)  ^\frac{p }{2}  dx \\
&=   \int   \left(    \sum_{k=M}^N  (\Gamma_k^0)^{-{2}/{\alpha}} f(x,V_k^0)^2   \right)^{\frac{p }{2}}dx\\
&=   \int   \left( \left(     \sum_{k=M}^N  (\Gamma_k^0)^{-{2}/{\alpha}} \left( f(x,V_k^0)^\alpha \right) ^ \frac{2}{\alpha}   \right)^{\frac{\alpha}{2}}\right)  ^\frac{p}{\alpha} dx\\
&\leq \int \sum_{k=M}^N  (\Gamma_k^0)^{-{p}/{\alpha}} |f(x,V_k^0)|^p  dx,
\end{split}
\end{equation*}
where we used the facts that $\rho_k$ are independent and $\ell_{q}$-norm is smaller 
than $\ell_{q'}$ norm when $0<	q'<1$ is less than $q$. Moreover, $\int |f(x,V_k^0)|^p dx< \infty$, almost surely, since its expectation is finite, and  the series  $\sum_{k=1}^\infty   (\Gamma_k^0)^{-{p}/{\alpha}}$ converges almost surely  by Lemma \ref{lem:poisson}.

Let us then assume $p>2$.  By Jensen's inequality, 
\begin{equation*}
\mathbb E\left[ \left \Vert \sum_{k=M}^N \rho_k (\Gamma_k^0)^{-{1}/{\alpha}} f(\cdot , V_k^0) \right \Vert_{L^p(T)}\right]^2 \leq 
\mathbb E\left[  \int \left| \sum_{k=M}^N \rho_k(\Gamma_k^0)^{-{1}/{\alpha}}f(x,V_k^0) \right|^p  dx \right]^\frac{2}{p}.
\end{equation*}
The order of expectation and integration can be changed by Fubini's theorem. We apply Khinchine's inequality to 
\begin{equation*}
\mathbb E \left[ \left| \sum_{k=M}^N \rho_k (\Gamma_k^0)^{-{1}/{\alpha}} f(x,V_k^0) \right|^p\right]  \leq c \left(\sum_{k=M}^N  
 (\Gamma_k^0)^{-{2}/{\alpha}}|f(x, V_k^0)|^2\right)^\frac{p}{2}.
\end{equation*}

Application of Minkowski's inequality with index $\frac{p}{2}$ produces the needed estimate for the limit \eqref{eqn:forT} through
\begin{equation*}
\begin{split}
\left( \int  \left(\sum_{k=M}^N  
 (\Gamma_k^0)^{-{2}/{\alpha}} |f(x,V_k^0)|^2\right)^\frac{p}{2}  dx\right)^\frac{2}{p} & \leq 
\sum_{k=M}^N      \left(  \int  (\Gamma_k^0)^{-{p}/{\alpha}} |f(x,V_k^0)|^p  dx \right)^\frac{2}{p} 
 \\  &= \sum_{k=M}^N   (\Gamma_k^0)^{-{2}/{\alpha}} \left(\int | f(x,V_k^0)| ^p  dx\right)^\frac{2}{p},
\end{split}  
  \end{equation*}
where
\begin{equation*}
\lim _{M,N\rightarrow \infty }\sum_{k=M}^N    (\Gamma_k^0)^{-{2}/{\alpha}} =0,
\end{equation*}
almost surely since the series  $\sum_{k=1}^\infty   \Gamma_k^{-{2}/{\alpha}}$ converges  almost surely
by Lemma \ref{lem:poisson}.  The expectation of the integral is finite according to the assumptions, 
so the integral is almost surely finite.

At this point,  we know that  $U\in L^p(T)$ almost surely. In  Remark  \ref{remark:pp}, we observed that  it is enough to show that $U$ is weakly measurable to verify that 
$U$ is $L^p(T)$-valued random variable.  In our case $\langle U,g_k\rangle_{L^p, L^q}$,  where $g_k\in L^q(T)$ and  $1/p +1/q =1$, are limits of truncated random series 
\begin{equation*}
 \langle U, g_k\rangle_{L^p, L^q} = 
 \lim_{N\rightarrow \infty } 
 C_\alpha ^{1/\alpha} \sum _{k=1}^N \rho_k \Gamma_k^{-1/\alpha} 
\langle f(\cdot ,V_k),g_k\rangle _{ L^p, L^q},
\end{equation*}
by continuity of the  linear forms $g_k$  on $L^p(T)$. Hence, they are measurable.

Let $U_1$ and $U_2$ be independent copies of $U$. Then 
\begin{equation*}
\begin{split}
\mathbb E [e^{i \langle aU_1+bU_2,g\rangle}]
 &= \exp\left( - \int \left| \int a f(x,x')g(x) dx \right|^\alpha + \left| \int b  f(x,x')g(x) dx \right|^\alpha dx' \right)\\
 &= 
  \exp \left( -  (a^{\alpha}+b^{\alpha})  \int \left| \int  f(x,x')g(x) dx \right|^\alpha dx'
  \right),
  \end{split}
\end{equation*}
by Theorem \ref{thm:rrep}.
Choosing  $a=b=2^{-1/\alpha}$ for all $g$ in the dual space of $L^p(T)$
 leads to the distribution of $U$. 
 According to the definition \eqref{eq:stable_def},   $U$ is then stable on $L^p(T)$ and, moreover, 
also symmetric.

  \end{proof}
 
 The stable random sheets are defined with bounded functions $f$. An immediate consequence of the previous theorem is $L^p$-regularity of the random sheets.
 \begin{cor}
 The stable  random sheets $U$  defined by \eqref{eqn:sheet} are  $ L^p([0,1]^d)$-valued random variables  for   $1\leq p<\infty$. 
 \end{cor}
 {
 The space $L^p([0,1]^d)$ is not an ideal sample space, since the topology of 
$L^p([0,1]^d)$ is too weak to  guarantee measurability of the  pointwise values  $U(x)$ for fixed $x\in[0,1]^d$. Since we already know that pointwise values are well-defined, we next study how  regular the sample paths can be.} {It is well-known that sample paths of $U(x)=\int f(x,x') M(dx')$ have discontinuities. Indeed, 
 a necessary condition for almost surely sample path continuity at $x$  is that $x'\mapsto f(x,x')$ is 
 continuous at $x$ (see Chapter 10.3 in \cite{TS94}), which clearly does not hold in our case.}

 Due to this we refine the analysis on the regularity of  stable random sheets  $U$ by considering Sobolev spaces.  We give a quick recap on $L^p$-based Sobolev spaces.  Recall, that the 
Schwartz space  $\mathcal S(\mathbb R^d)$ consists of all rapidly decreasing 
 smooth  functions. That is, 
 \begin{equation*}
 \mathcal S(\mathbb R^d) = \left\{ \lambda \in C^\infty (\mathbb R^d) : 
    \sup_{x\in\mathbb R^d }|  x^\alpha  D^\beta  \lambda (x) | <\infty  
    \text { for all multi-indices } \alpha, \beta \right \},
 \end{equation*}
where notations $x^\alpha= x^{\alpha_1}\cdots x^{\alpha_d}$ and 
$D^\beta =\partial_ 1 ^{\beta_1}\cdots \partial _d ^{\beta_d} $,  
$\alpha=(\alpha_1,\dots,\alpha_d), \beta=(\beta_1,\dots,\beta_d)$, are used with non-negative 
integers $\alpha_k, \beta_k$.
The space $\mathcal S(\mathbb R^d)$ is equipped with the seminorms 
$|\lambda  |_{\alpha,\beta} =   \sup_{x\in\mathbb R^d }|  x^\alpha  D^\beta  \lambda (x) | $, 
$\alpha,\beta \in  (\{0\}\cup \mathbb N)^d$.

The dual space of $\mathcal S(\mathbb R^d)$ is the space of tempered distributions 
$\mathcal S'(\mathbb R^d)$. The Fourier transform $\mathcal F( u)$  of a tempered distribution $u$  is defined by
\begin{equation*}
\langle \mathcal F ( u),  \overline{\lambda }\rangle =  (2\pi)^d \langle u, \overline {\mathcal F^{-1} (\lambda )} \rangle,
\end{equation*}
for all $\lambda \in \mathcal S(\mathbb R^d)$, where 
\begin{equation*}
\mathcal F( \lambda ) (\xi)=\int \exp \left( - i \xi \cdot x\right) \lambda (x) dx, \mathcal 
F ^{-1} (\lambda )(x) = \frac{1}{(2\pi)^d} \int \exp \left( - i \xi \cdot x\right) \lambda (x) dx,
\end{equation*}
and  $\overline{z} $ denotes the complex conjugate of $z\in\mathbb C$.  Let $1< p <\infty$ and $s\in\mathbb R$. The fractional order Sobolev  space $H^s_p$ consists of 
 the tempered distributions $u$ that satisfy
 \begin{equation}\label{eqn:sobolev_def}
 \mathcal F^{-1} \left( (1+|\xi|^2)^{\frac{s}{2}} \mathcal F(u) (\xi)\right) \in L^p(\mathbb R^d).
 \end{equation}
 and it is equipped with the norm  $\Vert  \mathcal F^{-1}\left( (1+|\xi|^2)^{\frac{s}{2}} \mathcal F(u) (\xi)\right)\Vert _{L^p}$. {To the best of our best knowledge, the following corollary is new given the refined analysis on the regularity}.
 \begin{cor} \label{thm:sobolev}
   Let  $s\geq 0, \  2\leq p<\infty $,  and  $f:\mathbb R^d\times E \rightarrow \mathbb C $ be Borel measurable. 
   
   If
$ \int_E \Vert  f(\cdot,x')\Vert_{H^s_p}^p dx'  <\infty $, then 
the symmetric stable random field
 \begin{equation}\label{eqn:series2}
U(x)=(C_\alpha |E|)^{{1}/{\alpha}} \sum_{k=1}^\infty \rho _k  \Gamma_k^{-{1}/{\alpha}} f(x, V_k),
  \end{equation}
 is an $ H^{s}_p$-valued random variable.
\end{cor} 
\begin{proof}
It is well-known that  $H^s_p$  is a subset of  $L^p$. Hence, 
$ \int _E \int  |f(x,x')|^p dx  dx' <\infty$.  Then  $U$ is $L^p$-valued random variable by Theorem  \ref{lem:LPpaths}. 
 
 Note that convergence of the series representation  \eqref{eqn:series2}  in $L^p(\mathbb R^d)$ implies  its  convergence in the space of tempered  distributions   $\mathcal S' (\mathbb R^d)$. Hence, we may change the order of  Fourier transform and   sum by the continuity of  Fourier transform on 
 $\mathcal S'(\mathbb R^d)$
Then 
\begin{equation} \label{eqn:series3}
\mathcal F U = 
(C_\alpha |E|)^{1/\alpha} \sum_{k=1}^\infty \rho_k \Gamma_k ^{-{1}/{\alpha}}  
 \mathcal F f(\cdot, V_k), 
 \end{equation}
 which by Theorem \ref{lem:LPpaths} converges in $L^q(\mathbb R^d)$ when $1/q+1/p=1$, because   the  Fourier transform maps  continuously from $L ^q $ into 
 $L^p$  by Hausdorff-Young inequality. Multiplying the $L^q$-function 
 \eqref{eqn:series3} with 
 $(1+|\xi|^2)^{s/2}$ produces a tempered distributions, whose 
 inverse Fourier transform is by continuity 
 \begin{equation} \label{eqn:series4}
\mathcal F^{-1} ((1+|\xi|^2)^\frac{s}{2} \mathcal F U) = 
(C_\alpha |E|)^{1/\alpha} \sum_{k=1}^\infty \rho_k \Gamma_k ^{-{1}/{\alpha}}  
\mathcal F^{-1 } (1+|\xi|^2)^{-\frac{s}{2}}\mathcal F f(\cdot, V_k),
 \end{equation}
in the space of tempered distributions.  But now
 the random series \eqref{eqn:series4} converges also in $L^p(\mathbb R^d)$ by  Theorem
 \ref{lem:LPpaths} and   Hausdorff-Young inequality applied for $\mathcal F^{-1}$.
\end{proof}

\begin{remark}
{As a consequence of Corollary \ref{thm:sobolev}, we have 
a function space formulation of symmetric stable random fields
that also cover pointwise values. Indeed, 
the pointwise evaluation $ U(x)$  are random variables 
 whenever $ U\in H^s_p(\mathbb R^d)$, where  $s> d/2$
 This fact follows  by writing the  pointwise evaluation  with the help of 
 Dirac delta function  $U(x)=\langle \delta_x, U\rangle_{H^{-s}_q , H^s_p}$, since 
 the dual space of $H^s_p(\mathbb R^d)$ is $H^{-s}_q(\mathbb R^d)$ and 
 $\delta _x \subset 
  H^{-d+d/p}_{p,}$. 
 Moreover, $$
 \Vert  \mathcal F^{-1} (1+|\xi|^2)^{-\frac{s}{2}}\exp( -i x \cdot \xi ) \Vert _{L^2 }\leq 
C\left( \int  \left|  (1+|\xi|^2)^{-\frac{s}{2}} \exp( -i x \cdot \xi ) \right| ^2 d\xi  \right)^\frac{1}{2}
<\infty,
  $$
  whenever $s>d/2$. 
for mappings $\langle \delta_x, U \rangle_{H^{-s}, H^s} : f\mapsto  f(x)$
that  are continuous. 
}
\end{remark}

{We now aim to extend all sample 
 paths  on $[0,1]^d$ to sample paths on $\mathbb R^d$ by zero extensions. That is 
 \begin{equation*}
 U^e(x)=\begin{cases}
  U(x),   \text{ if }x\in[0,1]^d\\
  0  \text{ otherwise.}
 \end{cases}
 \end{equation*}
We can establish this by multiplying the function $f(x,x')$ with $1_{[0,1]^d}(x)$.   }

\begin{thm}\label{thm:pathsob}
Let $2\leq p<\infty$.  
The $\alpha$-stable random sheet $U$ on $[0,1]^d$ has zero extension  
onto $\mathbb R^d$, which  is $H^s_p$-valued random variable
  for all  $ s <  \frac{1}{p}$.  
\end{thm} 
 
 \begin{proof}
 We  denote the  standard basis vectors of $\mathbb R^d$ with $e_1,\dots, e_d$.
   The Fourier transform of 
 the individual term   $1_{[0,x_\ell]}(V_k \cdot e_\ell)  =1_{(V_k\cdot e_\ell ,1]}(x_\ell)$, $\ell\in\{1,\dots, d\}$ with 
 respect to $x_\ell$  is  
 $\frac{\exp(- i\xi_\ell ) - \exp(i \xi_\ell  V_k \cdot e_\ell )}{-i \xi_\ell}$, which we insert 
 into the   Fourier transform  
 \begin{equation}\label{eqn:series5}
  \mathcal FU ^e (\xi)=  (C_\alpha |E|)^{1/\alpha} \sum_{k=1}^\infty \rho_k \Gamma_k^{-\frac{1}{\alpha}} 
 \prod _{\ell=1}^d  \frac{\exp(- i\xi_\ell ) - \exp(i \xi_\ell  V_k \cdot e_\ell)}{-i \xi_\ell},
 \end{equation}
 in $\mathcal S'(\mathbb R^d)$. 
According to Theorem \ref{lem:LPpaths}, the convergence in $L^p$  of 
\eqref{eqn:sobolev_def}, that is, 
\begin{equation*}
\begin{split}
\mathcal F^{-1} \left(  (1+|\xi|^2)^\frac{s}{2} \sum_{k=1}^\infty \rho_k \Gamma_k^{-\frac{1}{\alpha}}  \prod_{\ell=1}^d  
  \frac{\exp(- i\xi_\ell ) - \exp(i\xi_\ell  V_k \cdot e_\ell)}{-i \xi_\ell}\right)=\\
    \sum_{k=1}^\infty \rho_k \Gamma_k^{-\frac{1}{\alpha}} 
  \mathcal F^{-1}\left(  (1+|\xi|^2)^\frac{s}{2} 
 \prod_{\ell=1}^d   \frac{\exp(- i\xi_\ell ) - \exp(i \xi_\ell  V_k \cdot e_\ell)}{-i \xi_\ell}\right),
\end{split}
\end{equation*}
 depends on behaviour of the function
 \begin{equation*}
g(x,x'):= \mathcal F^{-1} \left( (1+ |\xi|^2 ) ^\frac{s}{2} \prod_{\ell=1}^d \frac{\exp(- i\xi_\ell ) - \exp(i \xi_\ell  x ' \cdot e_\ell  )}{-i \xi_\ell}\right) (x),
 \end{equation*}
 and we need to show that 
\begin{equation*} 
\int_{\mathbb R^d } \int_E  |g(x,x')|^p dx' dx  <\infty. 
\end{equation*}
Assume first that $2<p<\infty$.
By Hausdorff-Young inequality, the  inverse Fourier transform $\mathcal F^{-1}:L^{q}(\mathbb  R^d)\rightarrow 
L^p(\mathbb R^d)$ continuously for $2<p<\infty$ and $q^{-1}+p^{-1}=1$. Hence,
\begin{equation} \label{eqn:fur1}
  \left(\int |g(x,x')|^p d x \right) ^\frac{1}{p}  \leq C  
 \left( \int 
 (1+ |\xi|^2 ) ^\frac{sq}{2} \prod_{\ell=1}^d \left\vert  \frac{\exp(- i\xi_\ell ) - \exp(i \xi_\ell  x'
 \cdot e_\ell  )}{-i \xi_\ell}\right\vert ^q d\xi \right) ^\frac{1}{q}. 
\end{equation}
We will examine each coordinate $\xi_\ell$ for  values $|\xi |\leq 1 $ and $|\xi_\ell|>1$. 
Application of  Jordan's inequality produces estimate 
\begin{equation}\label{eqn:fur2}
\begin{split}
\prod_{\ell=1}^d\left\vert  \frac{\exp(- i\xi_\ell ) - \exp(i \xi_\ell  x'\cdot e_\ell  )}{-i \xi_\ell}\right\vert 
=  &\prod_{\ell=1}^d  \frac{ \left\vert  2\sin(\xi_\ell (1-x' \cdot e_\ell)/2)\right\vert } { |\xi_\ell |}
\\
\leq &
\prod_{\ell=1}^d  \left(  |1-s\cdot e_ \ell |     1_{|\xi_\ell | \leq 1}  +  \frac{2}{|\xi_\ell |} 1_{|\xi_\ell |>1}
\right),
\end{split}
\end{equation}
where   $x' \in[0,1]^d$ are uniformly  bounded.   We will obtain finite sum of products of the type 
\begin{equation}\label{eqn:fur3}
C(1+|\xi|^2)^\frac{sq}{2}\left( \prod_{\ell \in I_1  }    1_{|\xi_\ell | \leq 1} \right)  \left( \prod_{\ell\in I_2}  \frac{1}{|\xi_\ell |^q} 1_{|\xi_\ell |>1}
\right),
\end{equation}
where $I_1, I_2$ are disjoint sets such that $I_1\cup I_2=\{1,\dots,d\}$. 
Then the  products of denominators
\begin{equation*}
\prod_{\ell\in I_2 }  \frac{1}{|\xi_\ell| } 1_{|\xi_\ell|>1} 
 \leq  \prod_{\ell\in I_2 } \frac{\sqrt{2}}{(1+|\xi_\ell |^2)^\frac{1}{2}}.
 %\leq  \frac{2^\frac{d}{2}} {(1+ \sum_{\ell\in I_2 }|\xi_\ell |^2)^\frac{1}{2}}.
\end{equation*}

Moreover,
\begin{equation*}
\begin{split}
(1+|\xi|^2)^\frac{s}{2} \prod_{\ell\in I_1}    1_{|\xi_\ell | \leq 1}  
&\leq  2^{sd/2} (1+\sum_{\ell\in I_2}|\xi_\ell|^2)^\frac{s}{2}  \prod_{\ell\in I_1}    1_{|\xi_\ell | \leq 1}
\\ &\leq C  \prod_{\ell\in I_2} (1+ |\xi_\ell|^2)^\frac{s}{2} \prod_{\ell\in I_1}    1_{|\xi_\ell | \leq 1}.
\end{split}
\end{equation*}
 The upper bound for the integrand consists of finite  sum of products of the form 
\begin{equation*}
C \left( \prod_{\ell\in I_2} (|1+|\xi_{\ell}|^2| )^\frac{sq-q}{2} \right) 
\left( \prod_{\ell\in I_1 }1_{|\xi_\ell| \leq 1}\right),
\end{equation*}
which have finite integrals  when $sq - q < -1  \Leftrightarrow s < 1/p$, where $2< p<\infty$.
When $p=2$, we use  continuity of the Fourier transform and inverse Fourier transform on
$L^2(\mathbb R^d)$.
 \end{proof}
We emphasize that continuity of pointwise values do not hold in this topology due to jumps which can arise. For later purposes, we point out a refinement  of the convergence  in 
Theorem \ref{lem:LPpaths}. We formulate this result in a form of a lemma.
\begin{lem}\label{lem:uniformser}
Let $B\subset L^p(T)$ be a bounded set, where  $p>\max(1,\alpha)$.   Then the  limit  
of conditional expectations 
\begin{equation*}
\lim _{m\rightarrow \infty }\mathbb E \left [   \left \Vert \sum_{k=m }^\infty \rho_k \Gamma_k^{-1/\alpha} f(x,V_k)  \right  \Vert  \bigg |  \{\Gamma_k\}, \{ V_k\}  \bigg | \right ] =0,
 \end{equation*}
in $L^p(T)$ is  uniform with  respect to  $f\in B$.  
\end{lem}
\begin{proof}
We proceed exactly as in the proof of  Theorem \ref{lem:LPpaths}, which  indicates that the corresponding Cauchy property is  uniform with respect to $f\in B$. Hence the convergence is uniform. 
\end{proof}

\section{Posterior analysis}
\label{sec:conv}

In this section, we will fix  the  formulation of the Bayesian inverse problem with $\alpha$-stable sheets as priors. We will proceed by aiming to approximate the $\alpha$-stable sheets with finite dimensional random sheets. This will lead to us showing that the posterior distributions are consistent with respect to an increasing dimensionality of the approximations of the prior.

\subsection{The setup of the Bayesian inverse  problem}

Let $G$ and $\widetilde G$ be separable  Banach spaces and let the forward model  be
$K:F\times \widetilde G\rightarrow G$, where   the space $F$, where the prior lives, is   either  $L^p(T)$, where $T=[0,1]^d$ $1\leq p <\infty$  and  $\alpha <p$   or ${H^{s}_p(T)}$,  where $2\leq p<\infty$, $0< s <\frac{1}{p}$.  The noisy observation of the unknown $U$ is modelled as 
$Y=K(U,\eta)$, where $\widetilde G$-valued random variable $\eta $ represents noise, which we take to be statistically independent of $U$.   

From the posterior distribution 
\begin{equation*} 
\mu^y \propto \exp(-\Phi(u,y)) \mu(du),
\end{equation*}
the prior distribution $\mu$ is taken to be the distribution of an $\alpha$-stable random sheet $U$
on $F$  and  the NDLL is defined as
 \begin{equation*}
 \Phi(u,y)= -\ln \left(  \frac{d\mu_{K(u,\eta)}}{d\nu}(u)\right),
 \end{equation*}
  for some distribution $\nu$ on $G$.  Now if we take $G,\widetilde G=\mathbb R^k$ and 
  $
  K(u,\eta )= Lu +\eta,
  $
  with $\eta \sim N(0,\Sigma)$ on $\mathbb R^k$ and $L: F\rightarrow \mathbb R^k$  a continuous linear 
  mapping. Then 
  \begin{equation*}
\mu^y \propto  \exp\left (-\frac{1}{2} | y- Lu|^2_\Sigma \right) \mu(du),  
    \end{equation*}
where the norm $| \cdot |_\Sigma = | \Sigma^{-\frac{1}{2}} \cdot | $.  In this case,  $\Phi$ is bounded on bounded subsets of $F\times G$, continuous on $F\times G$ and uniformly continuous as a function of $y$  with respect to $u$ in bounded subsets of $F$. All normalizing constants are bounded, since 
$\Phi$ is non-negative.  Hence $\Phi$ satisfies Conditions WD1, WD2, WP1, WP2 and WP3 with 
$D_\mu=\mathbb R^k$ and the posterior distribution  $\mu^y$ is well-posed on $\mathbb R^k$ in  total variation metric by Theorem \ref{thm:TV}.

\subsection{Approximations of the prior}

Up until now we have discussed ways of representing $\alpha$-stable sheets
\begin{equation}\label{eqn:prior}
U(x) = (C_\alpha) ^{1/\alpha} \sum_{k=1}^\infty 
\rho_k \Gamma_k^{-1/\alpha} f(x,V_k), 
\end{equation}
where $x\in[0,1]^d$.  We now wish to construct a numerically applicable  discretized version which can be applied e.g.\ in MCMC sampling.

A natural discretization of \eqref{eqn:prior}  arises by using  a random walk type approximation with 
independent increments.    We consider  a uniform grid $\{ x=mh :  m\in \{0,\dots, N\}^d\}$,  $h=1/N$ and $N\in \mathbb N$.   We equip the index set 
$\{m\in \{0,\dots, N\}^d\}$ with partial order $\leq  $  defined by   $(k_1\dots, k_d)\leq  (m_1,\dots, m_d)$ if and only if $k_i\leq  m_i$ for all $i=1,\dots,d$. Similarly, we set 
$(k_1\dots, k_d)< (m_1,\dots, m_d)$ if and only if $k_i<  m_i$ for all $i=1,\dots,d$.
 
 The discretization of $U$ is based on representing the
 values of $U$ at gridpoints as sums of independent increments, namely
\begin{equation}
\label{eqn:approximation}
U(hm)= \sum_{0<  n \leq m} \int 1_{C_{n}}(x) M( dx),
\end{equation}
where $C_n$ are  hypercubes, which are the translations of $(0,h]^d$ by gridpoints $h(n-1)$, where
 $n-1$ has components $n_i-1$ for all $i=1,\dots,d$.  Each hypercube $C_n$ is  of Lebesgue measure $|C_n| = h^d$. Since the hypercubes $C_n$ are disjoint,
 the stochastic integrals $\int 1_{C_{n}}(x) M( dx)$ are independent (see Theorem 3.5.3 in \cite{TS94}). 

The discretization of  $U$ on $[0,1]^d$ is taken to be 
\begin{equation}\label{eqn:discrete}
U^N(x)=U(h \lceil x/h \rceil ),  
\end{equation}
 where the ceiling function $\lceil x \rceil =  \min \{ m\in \mathbb Z^d:  x_j \leq m_j , \; j=1,\dots, d\}$. 
  Notably $U^N$ has piecewise constant sample paths, whose values   can be expressed as a sum of  independent identically distributed symmetric $\alpha$-stable random variables
\begin{equation*}
U^N(x)= \begin{cases} 
\sum_{0< n \leq  \lceil x/h \rceil }  \int 1_{C_{n}}(x') M( dx') \text{ if } x\in (0,1]^d\\
0 \text{ otherwise}.
\end{cases}
\end{equation*}
  This makes it easy to express the probability density of $U^N$ for $\alpha=1$ and, consequently, generate samples of $U^N$ in MCMC methods. A useful way of describing the discretized random sheet is by replacing the sum of independent 
increments  with  difference equations. In particular we are primarily interested in discretized version of Cauchy random sheets. 
For simplicity, we will discuss the case $d=2$, where  values  $U(hm)=U(hm_1,hm_2)$ are 
represented by double sums 
\begin{equation}\label{eqn:dis2d}
U(hm_1,hm_2)= \sum_{n_1=1}^{m_1} \sum_{n_2=1}^{m_2} \int 1_{C_{(n_1,n_2)}}(x) M( dx).
\end{equation}
Taking  the differences of values \eqref{eqn:dis2d} with respect to the both coordinates  leads us to 
equations
\begin{align}
\label{eqn:cdp_dis}
 U(hm_1, hm_2) &- U(hm_1, h(m_2-1))- U (h(m_1-1), hm_2)   \\
 \nonumber
 &+ U (h(m_1-1),h(m_2-1) ) = \int 1_{C_{(m_1,m_2)}}(x) M( dx),
\end{align}
with zero boundary values on the coordinate axes. In Equation \eqref{eqn:cdp_dis},  the 
random variables $U(hm_1,hm_2)$  conditioned with random variables  $ U(hm_1, h(m_2-1),  U (h(m_1-1), hm_2),  $ and $(h(m_1-1),h(m_2-1)$ are  i.i.d.   with distribution $S_\alpha(h^{d/\alpha},0,0)$.  We mention that the difference equations  \eqref{eqn:cdp_dis} formally  have  a continuous counterpart 
\begin{equation*}
\partial_{x_1}\partial_{x_2}  U = W, 
\end{equation*}
where  Cauchy  noise  
$$
W :=  (C_\alpha) ^{{1}/{\alpha}} \sum_{k=1}^\infty 
\rho_k \Gamma_k^{-1/\alpha} \delta_{V_k},
$$
can be understood as a  distributional derivative of the 
Cauchy sheet.  However, we do not proceed in this direction, but focus on convergence of 
the approximations.

\subsection{$L^p$-convergence}

We now move to  the convergence analysis of \eqref{eqn:discrete} on $L^p([0,1]^d)$. Following on from Lemma \ref{lem:poisson} we are in a position to show convergence  of our approximations \eqref{eqn:discrete}. This is done through the following theorem.

\begin{thm}
\label{lem:first_pp}
Let $1\leq p<\infty$.  The approximations  $U^N(x)=U(h \lceil x/h \rceil )$ 
converge  to $ U $ on $L^p([0,1]^d)$ in distribution. 
\end{thm}
\begin{proof}
For $0<\alpha <1$, the summands in 
\begin{equation}
\label{eqn:sum}
U^N(x)=  C_\alpha ^{1/\alpha} \sum_{k=1}^\infty  \rho_k \Gamma_k^{- 1/\alpha}   1_{[V_k\cdot e_1,1 ]\times \cdots\times  [V_k\cdot e_d,1 ] }(h  \lceil x/h \rceil  ),
\end{equation}
are bounded by  $\Gamma_{k}^{-1/\alpha }$, which are almost surely summable by Lemma \ref{lem:poisson}.   By dominated convergence,  we may take the limit inside the $L^p$-norm and obtain
 \begin{equation*}
 \begin{split}
&\lim_{N\rightarrow \infty} \int  | U^N(x)- U(x)|^p dx  \\ = &C_\alpha ^{1/\alpha}  \int  \bigg|   \sum_{k=1}^\infty 
 \rho_k\Gamma_k^{- 1/\alpha}    \bigg( \lim_{N\rightarrow \infty} 1_{[V_k\cdot e_1,1 ]\times \cdots \times  [V_k\cdot e_d,1 ] }(h  \lceil x/h \rceil ) \\
   -&  1_{[V_k\cdot e_1,1 ]\times \cdots \times  [V_k\cdot e_d,1 ] }(x) \bigg) \bigg|^p dx   = 0,
 \end{split}
\end{equation*}
since the indicator functions in Equation \eqref{eqn:sum} are multidimensional generalisations of right-continuous functions  on $[0,1]^d$.  Next, we consider the case of $1\leq \alpha <2$.  We will show that a sufficient condition 
for weak convergence, namely  
\begin{equation*}
\lim_{N\rightarrow \infty} \mathbb E\left[ g(U^N)- g(U)\right] = 0,
\end{equation*}
for all bounded Lipschitz functions $g$  on $L^p([0,1]^d)$, holds (see Corollary 2.3.5 in \cite{VIB18}). 

It is enough to show that the  conditional expectations 
\begin{equation*}
\mathbb E\left[ g(U^N)- g(U) \big| 
 \{\Gamma_k\}, \{V_k \}   \right] ,
\end{equation*}
converge almost surely to zero, since  we may exchange the order of the  limit and the  expectation in 
\begin{equation*}
 \lim _{N\rightarrow \infty }   \mathbb E\left [\mathbb E\left[ g(U^N)- g(U) \big|
 \{\Gamma_k\},\{ V_k\} \right]\right]
   =  
 \mathbb E\left [ \lim _{N\rightarrow \infty } \mathbb E\left[ g(U^N)- g(U) \big|
 \{\Gamma_k\},\{ V_k\} \right]\right],
\end{equation*}
by monotonicity of the conditional expectation and the  boundedness of $g$. Since $g$ is a Lipschitz function, the difference
\begin{equation*}
\left| \mathbb E\left[ g(U^N)- g(U)  \big| 
 \{\Gamma_k\}, \{V_k \}  \right]  \right|  \leq 
 C \mathbb E \left [ \Vert U^N- U\Vert_p     \big| \{\Gamma_k\}, \{V_k \} \right],
\end{equation*}
reduces to conditional expectation of the $L^p$-norm. Recalling that 
$\rho_k$ are independent from $\Gamma_k$ and $V_k$ leads to 
\begin{equation*}
\left| \mathbb E\left[ g(U^N)- g(U)  \big|  \{\Gamma_k\}, \{V_k \}   \right] \right|     \leq  
C \mathbb E  \left [ \left \Vert \sum_{k=1}^\infty \rho_k (\Gamma_k^0)^{-1/\alpha} \left( 
f( h  \lceil x/h \rceil , V_k^0) - f(x,V_k^0)\right ) \right\Vert_{p} \right],
\end{equation*}
with fixed values of  $V_k=V_k^0$ and $\Gamma_k = \Gamma_k^0$. We will proceed by dividing  the Rademacher series  in two parts as in 
\begin{equation*} \begin{split}
&\left| \mathbb E\left[ g(U^N)- g(U)   \big|  \{\Gamma_k\}, \{V_k \}  \right]  \right| \\ &\leq   
C \mathbb E  \left [  \left \Vert \sum_{k=1}^m  \rho_k (\Gamma_k^0)^{-1/\alpha} \left( 
f( h  \lceil x/h \rceil , V_k^0)  - f(x,V_k^0)  \right)\right \Vert_{p} \right] \\
 & + 
\left  \Vert \sum_{k=m+1}^\infty  \rho_k (\Gamma_k^0)^{-1/\alpha} \left( 
f( h  \lceil x/h \rceil , V_k^0)  - f(x,V_k^0)  \right) \right \Vert_{p} \\
&=  : I_1+I_2.
\end{split}
 \end{equation*}
 Assume first that $p>\alpha$.  For fixed $m$, the terms $I_1=I_1(N)$ converge  to zero as $N$ grows unlimited. By Lemma \ref{lem:uniformser}, the term $I_2=I_2(m)$ converges to zero uniformly with respect to $N$ as $m$ grows unlimited, since functions  $f( h  \lceil x/h \rceil , V_k^0)- f(x,V_k)$  are bounded in $L^p([0,1]^d)$. For $1\leq p \leq \alpha$, the result follows from  the continuous   imbedding of  $L^s(T)$ to $ L^t(T)$, when  $s>t$. 
 
  \end{proof}

\begin{thm}
\label{thm:conv_UN}
Let  $\mu_N$  and $\mu$ be the distributions of $U^N$ and $U$ on $L^p([0,1]^d)$, respectively. 
Let a  NDLL $\Phi$ be non-negative and satisfy Conditions WD2 and PC1.     Then 
the posterior distributions $\mu^y_N \propto\exp(-\Phi(u,y))\mu_N(du)$ converge to $\mu^y\propto \exp(-\Phi(u,y)) \mu(du)$ in weak topology. 
\end{thm}
\begin{proof}
The posterior distributions  are  well-defined on the set $L^p([0,1]^d)$ by 
Theorem \ref{thm:WD}.  Moreover,
\begin{equation*}
\lim_{N\rightarrow \infty} \int \exp(-\Phi(u,y)) \mu_N(du)= 
\int \exp(-\Phi(u,y)) \mu(du),
\end{equation*}
since $\mu_N$ converge weakly to $\mu$ and $\exp(-\Phi(u,y))$ is a bounded continuous function function on $L^p([0,1]^d)$ for every $y\in G$.
\end{proof}

Another approximation arises from the truncated LePage series
\begin{equation}\label{eqn:Lepage1}
U^N(x):= (C_\alpha)^{1/\alpha}\sum_{k=1}^N \rho_k \Gamma_k^{-1/\alpha} f(x,V_k).
\end{equation}
Below we provide an analogous theorem of Theorem \ref{thm:conv_UN} while using \ref{eqn:Lepage1} as $U^N$.
\begin{thm}
Let  $\mu_N$  and $\mu$ be the distributions of \eqref{eqn:Lepage1} and $U$ on $L^p([0,1]^d)$ (or $H^s_p$), respectively.  Let a  NDLL $\Phi$ be non-negative and satisfy Conditions WD2 and PC1.  Then 
the posterior distributions $\mu^y_N \propto\exp(-\Phi(u,y))\mu_N(du)$ converge to $\mu^y\propto \exp(-\Phi(u,y)) \mu(du)$ in total variation metric.
\end{thm}

\begin{proof}
Since $\exp(-\Phi(u,y))$ are bounded continuous functions and  $U^N$ converges to $U$ 
almost surely by Theorem \ref{lem:LPpaths} (or Theorem \ref{thm:pathsob}), the result follows immediately by Lebesgue's dominated convergence theorem, when we write 
\begin{equation*}
\begin{split}
\sup_{A} |\mu^y_N(A)-\mu^y(A)|=  & \mathbb E \left[  \left|( Z_y^{N})^{-1}  \exp(-\Phi(U^N,y))- 
 Z_y^{-1} \exp(-\Phi(U,y)) \right| \right] \\
 \leq & \frac{\left| Z_y^{N}- Z_y\right|}{Z_y} +   (Z_{y})^{-1} \mathbb E\left[  \left|\exp(-\Phi(U^N,y))- 
 \exp(-\Phi(U,y)) \right| \right].
\end{split}
\end{equation*}
\end{proof}

{
\section{Conclusions}
{Our focus on this paper was to motivate an analytical understanding of $\alpha$-stable random fields through the application of Bayesian inversion. What we aimed to achieve was a well-posedness theorem and numerous convergence results for $\alpha$-stable  sheets. This was plausible with different representations that the random sheets could take.  Furthermore we were able to show that a discretized representation of the stable sheets remained consistent with the original form for increasing dimensions. The use and need of non-Gaussian priors for Bayesian inverse problems is apparent, and as a result this work leads to many other interesting directions of research:}}

\begin{itemize}
\item One avenue to take with this is to consider the numerical study and verification of $\alpha$-stable priors, where inference is done on the hyperparameters. This would include modifying the hyperparameters such as the stability parameter $\alpha$ depending on the underlying unknown and perhaps working in a hierarchical manner as done in \cite{CIRL17}.
\item From a Bayesian perspective understanding contraction rates \cite{HRSH15} of these priors poses useful insight. There has been extensive work on this with Gaussian priors which includes the inverse problem setting, \cite{ALS12,GGV00,VZ08}. For non-Gaussian priors some initial work has been done in the case of Besov priors \cite{ABDH17}. 
\item {Using the framework discussed in the paper could result in a way of modelling stable fields \cite{CLL08}. A natural application of this would be machine learning, where already there has been some work conducted on stable processes within neural networks \cite{DL06,RMN96}.}
\end{itemize}
These directions and more will be considered for future work.
\section*{Acknowledgements} 
NKC acknowledges a Singapore Ministry of Education Academic Research Funds Tier 2 grant [MOE2016-T2-2-135]. SL and LR were funded by Academy of Finland, grant numbers 326240 and 326341.
%}
\bibliographystyle{plain}

\end{document}